\documentclass[10pt]{article}

\usepackage[margin=26mm]{geometry}
\usepackage{amsthm}
\usepackage{amsmath}
\usepackage{mathbbol}
\usepackage{amsfonts}
\usepackage{amssymb}
\usepackage{graphicx}
\usepackage{algorithm}
\usepackage{url}
\usepackage{wrapfig}
\usepackage{color}

\usepackage{mathtools}
\usepackage{stmaryrd}

\usepackage{subfig}

\usepackage{algorithm}
\usepackage{algpseudocode}
\usepackage{comment}

\newcommand{\N}{\mathbb{N}}

\newcommand{\R}{\mathbb{R}}

\newcommand{\polytope}{\mathcal{P}}

\newcommand{\complex}{\mathcal{C}}
\newcommand{\liftedstartriangulation}{\mathcal{T}}
\newcommand{\triangulation}{\mathcal{T}}
\newcommand{\pointset}{V}

\newcommand{\pointsetnonconvex}{\pointset'}
\newcommand{\spineset}{U}
\newcommand{\spinesetcard}{n}
\newcommand{\simplotope}{\mathcal{S}}
\newcommand{\linearmap}{\Phi}

\newcommand{\parallelotope}[1]{p(#1)}

\newcommand{\hyperplane}{\mathbb{H}}
\newcommand{\rotationaxis}{\mathbb{A}}

\newcommand{\facet}{\mathcal{F}}
\newcommand{\ridge}{\mathcal{R}}
\newcommand{\face}{\mathcal{F}}

\newcommand{\pulling}{\mathrm{Pull}}

\newcommand{\simplex}{\Delta}

\newcommand{\birkhoff}{\mathcal{B}}
\newcommand{\birkhoffspine}{U}

\newcommand{\fold}[1]{\hat{#1}}

\newcommand{\ce}{:=}
\newcommand{\set}[1]{\left\{#1\right\}}
\renewcommand{\vector}[1]{{\mathbf{#1}}}

\newcommand{\abs}[1]{\left\lvert#1\right\rvert}

\newcommand{\everestpolytope}{\mathcal{E}}
\newcommand{\setransform}{\Pi}

\newcommand{\everestpolytopelong}[2]{{\everestpolytope_{#1,#2}}}
\newcommand{\verticesone}[2]{{P_{#1,#2}}}
\newcommand{\verticesminusone}[2]{{V_{#1,#2}}}
\newcommand{\verticeszero}[2]{{U_{#1,#2}}}
\newcommand{\verticeseverest}[2]{{E_{#1,#2}}}

\newcommand{\standard}{\vector{e}}

\newcommand{\conv}{\operatorname{conv}}

\newcommand{\SoVertices}{{V}}
\newcommand{\underlyingspace}[1]{\bigcup #1}

\newcommand{\vol}{\mathrm{vol}}

\newtheorem{maintheorem}{Main Theorem}

\newtheorem{theorem}{Theorem}
\newtheorem{lemma}[theorem]{Lemma}
\newtheorem{corollary}[theorem]{Corollary}

\title{Constrained Triangulations, Volumes of Polytopes,\\ and Unit Equations}

\author{
Michael Kerber, Robert Tichy, Mario Weitzer
}
\date{}

\begin{document}

\maketitle

\begin{abstract}
\noindent
Given a polytope $\mathcal{P}$ in $\mathbb{R}^d$ 
and a subset $U$ of its vertices,
is there a triangulation of $\mathcal{P}$ using $d$-simplices
that all contain $U$?
We answer this question by proving an equivalent and easy-to-check
combinatorial criterion for the facets of $\mathcal{P}$.
Our proof relates triangulations of $\mathcal{P}$ to
triangulations of its ``shadow'', a projection to a lower-dimensional
space determined by $U$. In particular, we obtain a formula
relating the volume of $\mathcal{P}$ with the volume of its shadow.
This leads to an exact formula for the volume of a polytope
arising in the theory of unit equations.
\end{abstract}

\section{Introduction}
\paragraph{Problem statement and results.}
Let $\polytope$ be a convex polytope in $\R^d$, that is,
the convex hull of a finite point set $\pointset$,
and let $\spineset$ be a subset of $\pointset$.
We ask for a triangulation of (the interior of) $\polytope$
with the property that every $d$-simplex in the triangulation
contains all points of $\spineset$ as vertices, calling
it a \emph{$\spineset$-spinal triangulation}. 
A simple example is the \emph{star triangulation} of $\polytope$ (Figure~\ref{fig:example_star_triangulation}),
where all $d$-simplices contain a common vertex $\vector{p}$,
and $\spineset$ is the singleton set consisting of that point.
Another example is the $d$-hypercube with $\spineset$ being a pair of opposite
points (Figure~\ref{fig:example_spinal_triangulation}). Indeed, the hypercube can be triangulated in a way that 
all $d$-simplices contain the space diagonal spanned 
by $\spineset$~\cite{Freudenthal:1942,EdelsbrunnerKerber:2012}.

\begingroup
\centering
\begin{minipage}[t][5cm][t]{0.475\textwidth}
\begin{figure}[H]
\centering
\includegraphics[height=2.75cm]{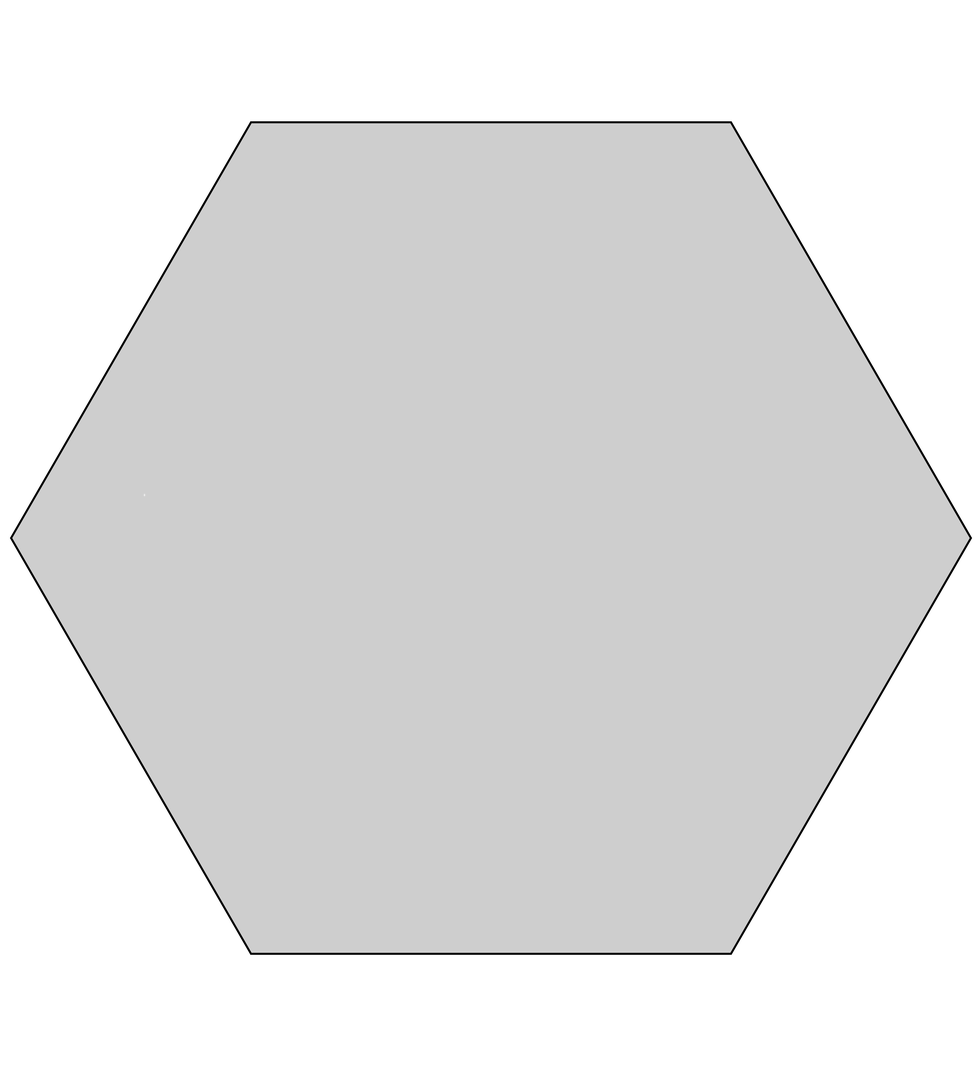}
\;
$\overset{\longrightarrow}{\rule{0cm}{1.25cm}}$
\;
\includegraphics[height=2.75cm]{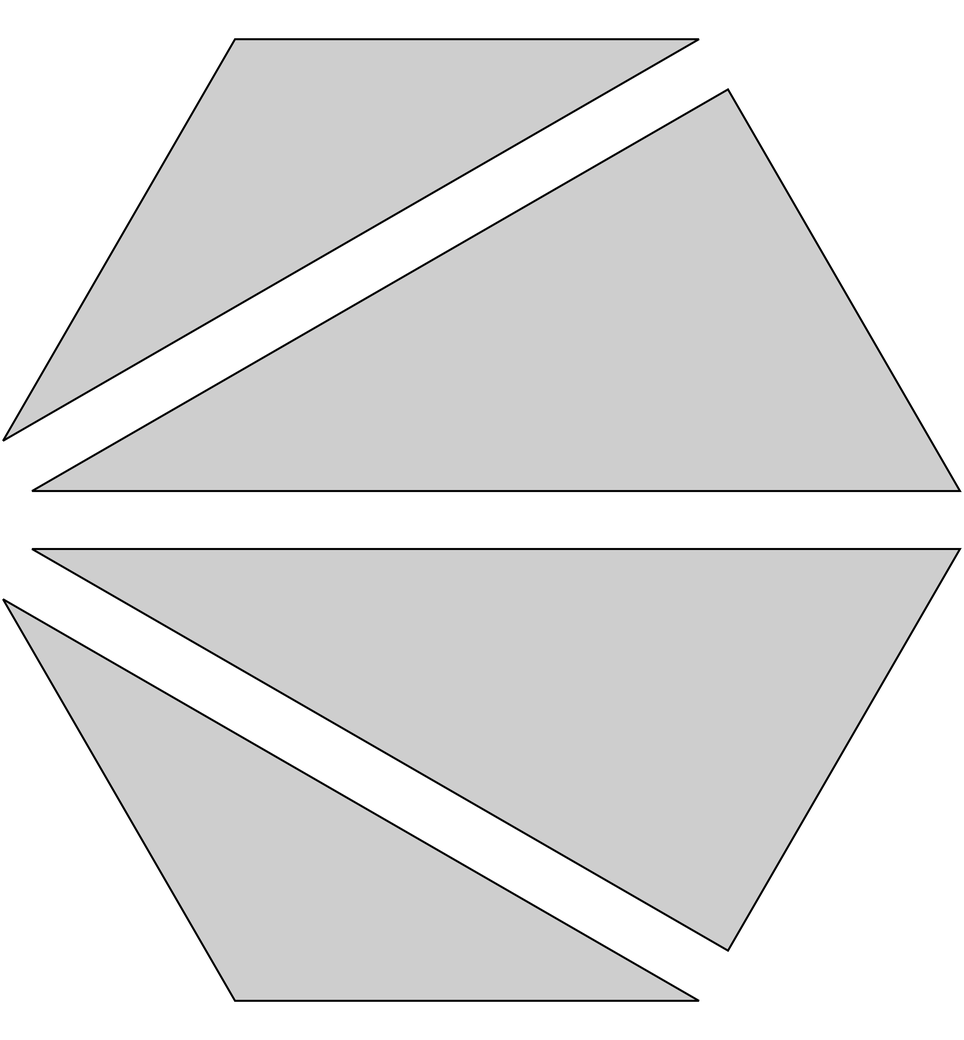}
\caption{A star triangulation of a hexagon.}
\label{fig:example_star_triangulation}
\end{figure}
\end{minipage}
\quad
\begin{minipage}[t][5cm][t]{0.475\textwidth}
\begin{figure}[H]
\centering
\includegraphics[height=2.75cm]{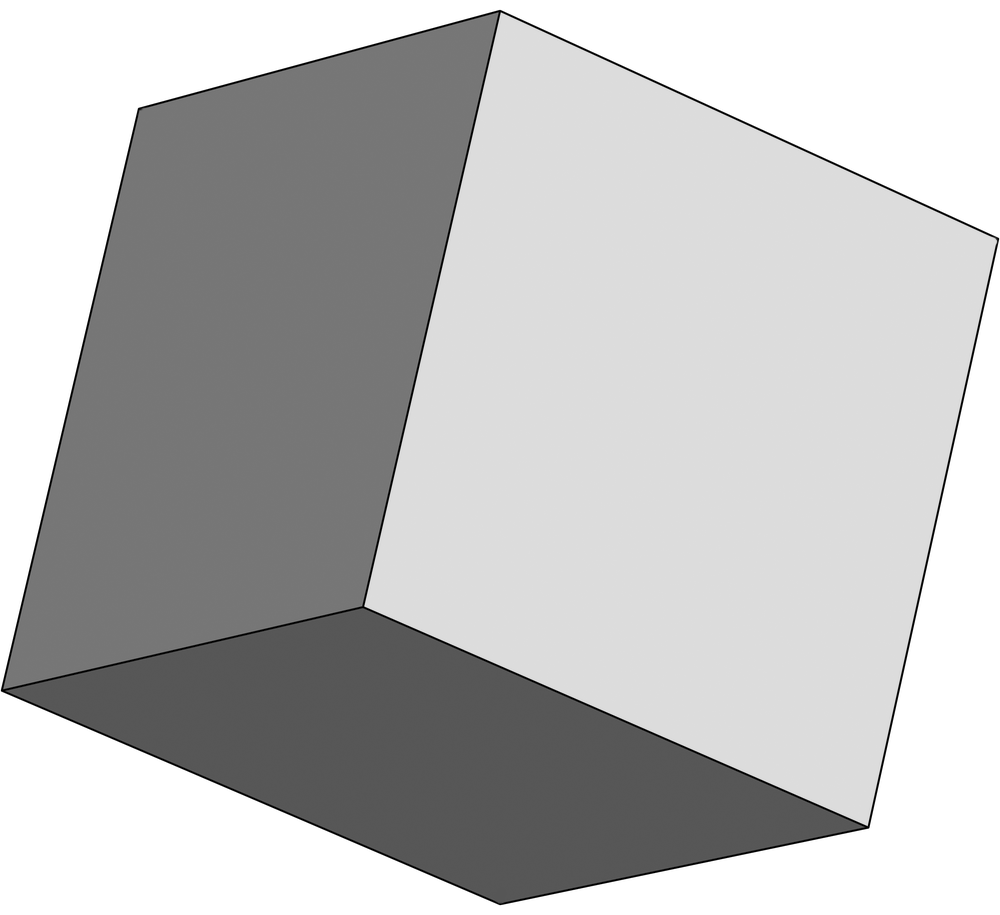}
\;
$\overset{\longrightarrow}{\rule{0cm}{1.25cm}}$
\;
\includegraphics[height=2.75cm]{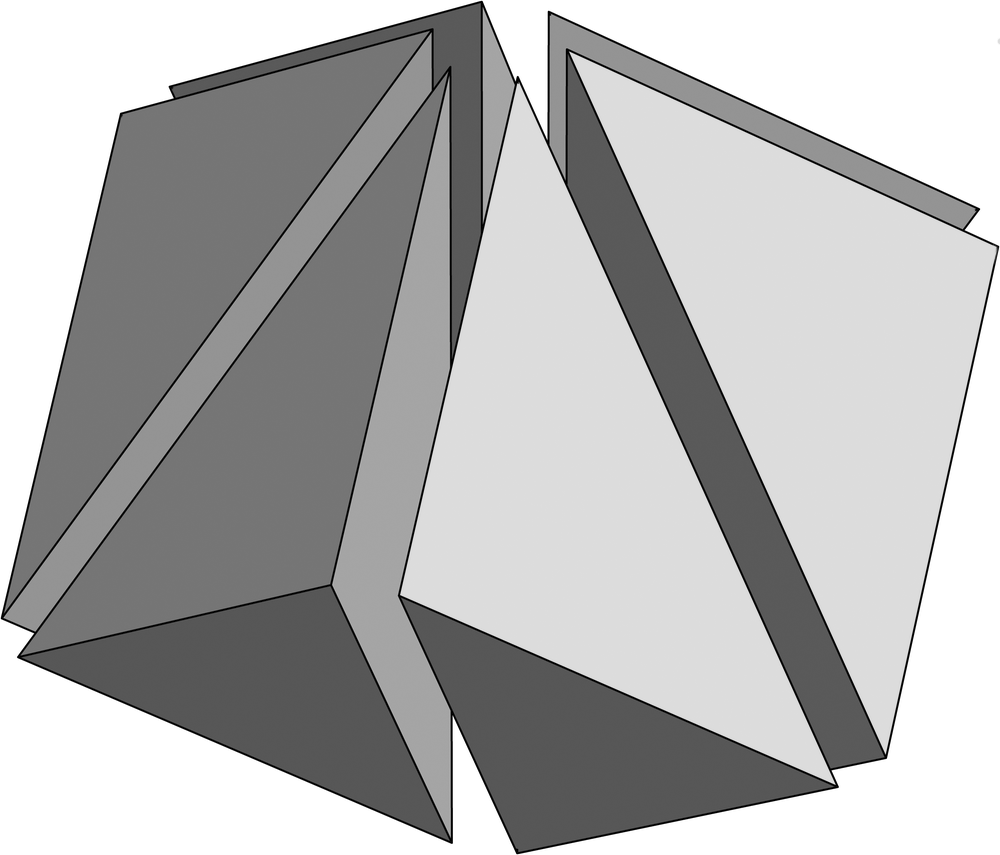}
\caption{A $\spineset$-spinal triangulation of a cube, where $\spineset$ consists of the two vertices on a space diagonal.}
\label{fig:example_spinal_triangulation}
\end{figure}
\end{minipage}
\endgroup

We are interested in what combinations of $\polytope$ and $\spineset$
admit spinal triangulations. Our results provide a simple
combinatorial answer for this question: 
Denoting by $\spinesetcard$ the cardinality of $\spineset$, 
a $\spineset$-spinal
triangulation of $\polytope$ exists if and only if each facet of $\polytope$
contains at least $\spinesetcard-1$ vertices of $\spineset$. 
In that case, we call $\spineset$ a \emph{spine} of $\polytope$.
More generally, we provide a complete characterization of spinal triangulations:
let $\linearmap$ denote the orthogonal projection of $\R^d$ to the orthogonal
complement of the lower-dimensional flat spanned by $\spineset$.
$\linearmap$ maps $\spineset$ to $\vector{0}$ by construction,
and $\polytope$ is mapped to a \emph{shadow} 
$\fold{\polytope}\ce\linearmap(\polytope)$. We obtain a $\spineset$-spinal triangulation
of $\polytope$ by first star-triangulating $\fold{\polytope}$ with respect
to $\vector{0}$ and then lifting each maximal simplex to $\R^d$ by taking 
the preimage of its vertices under $\linearmap$ (Figure~\ref{fig:example_lifting_process}). Vice versa, every spinal triangulation
can be obtained in this way.
\begin{figure}[H]
\centering
\includegraphics[height=2.75cm]{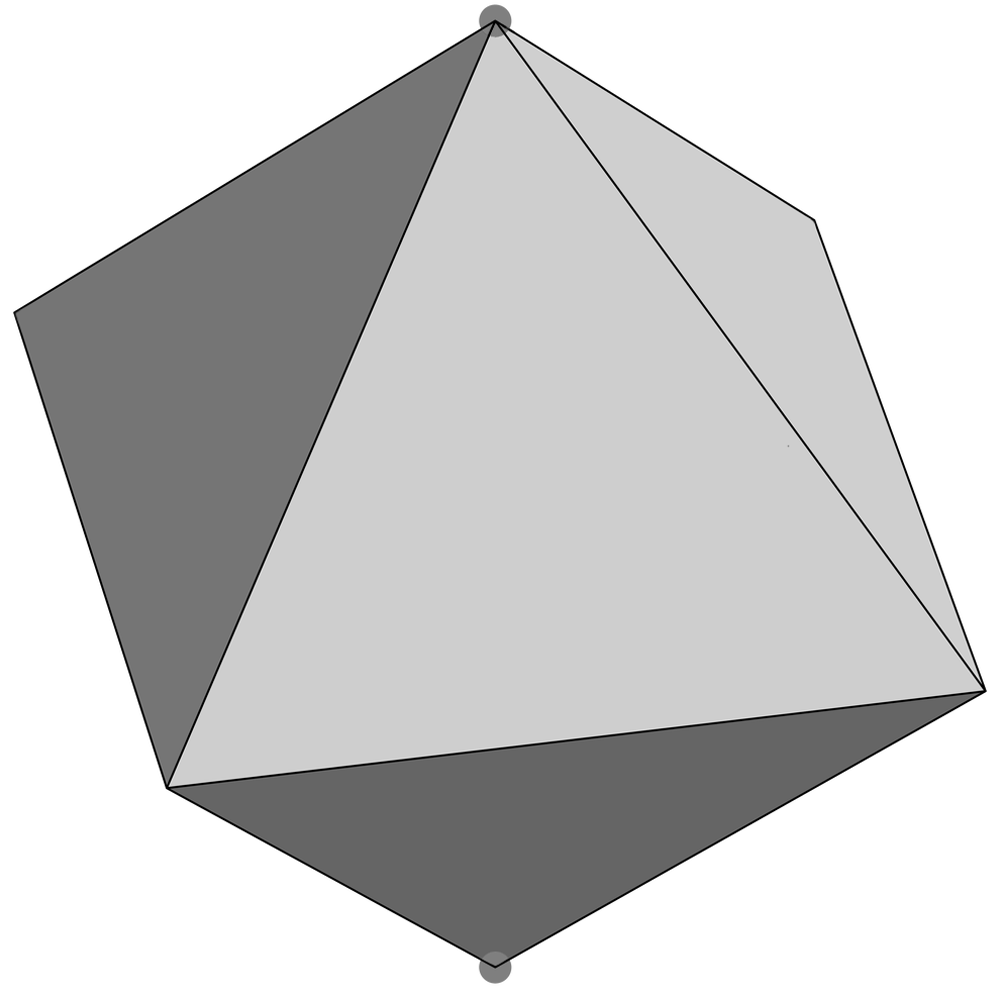}
\,
$\overset{\overset{(1)}{\longrightarrow}}{\rule{0cm}{1.25cm}}$
\includegraphics[height=2.75cm]{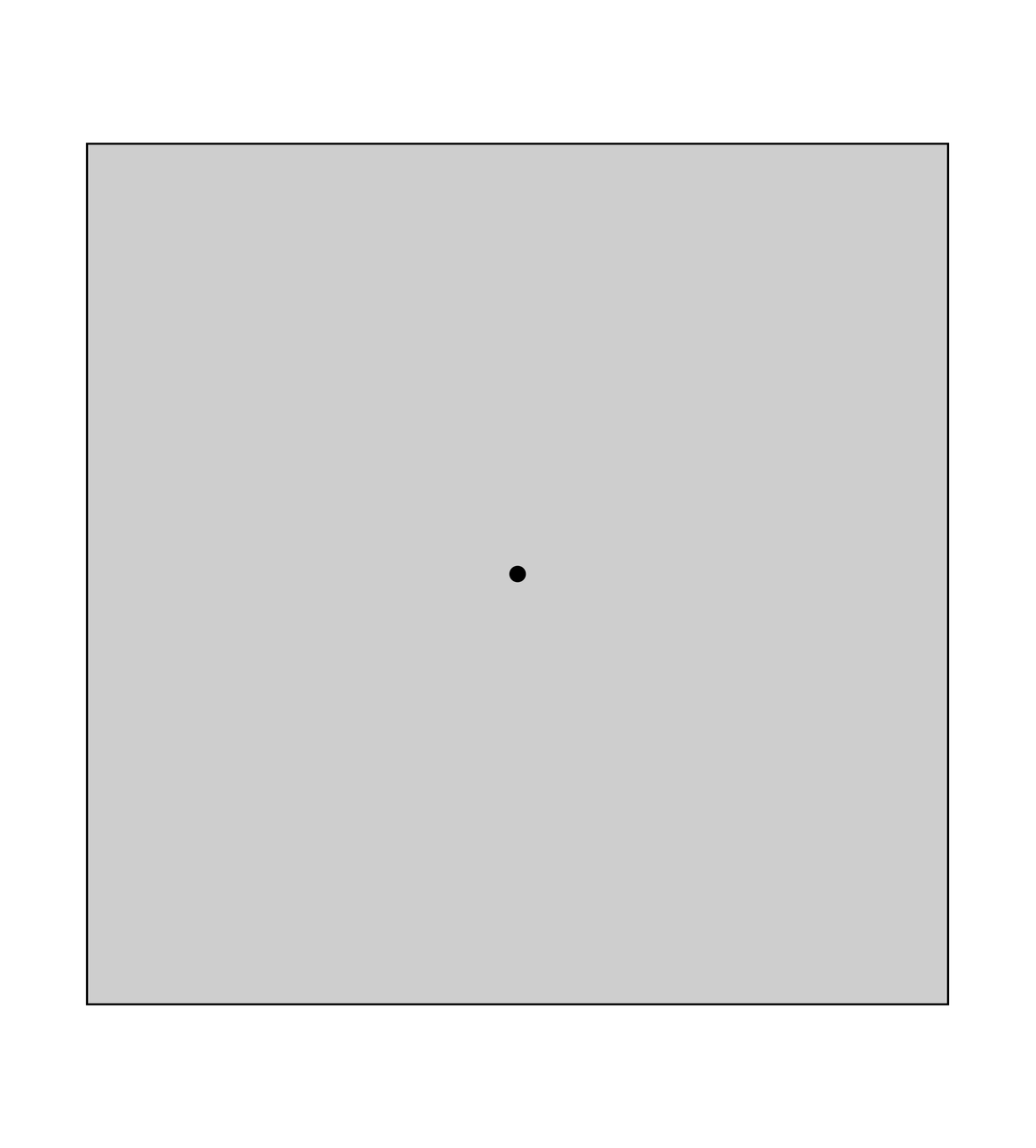}
$\overset{\overset{(2)}{\longrightarrow}}{\rule{0cm}{1.25cm}}$
\includegraphics[height=2.75cm]{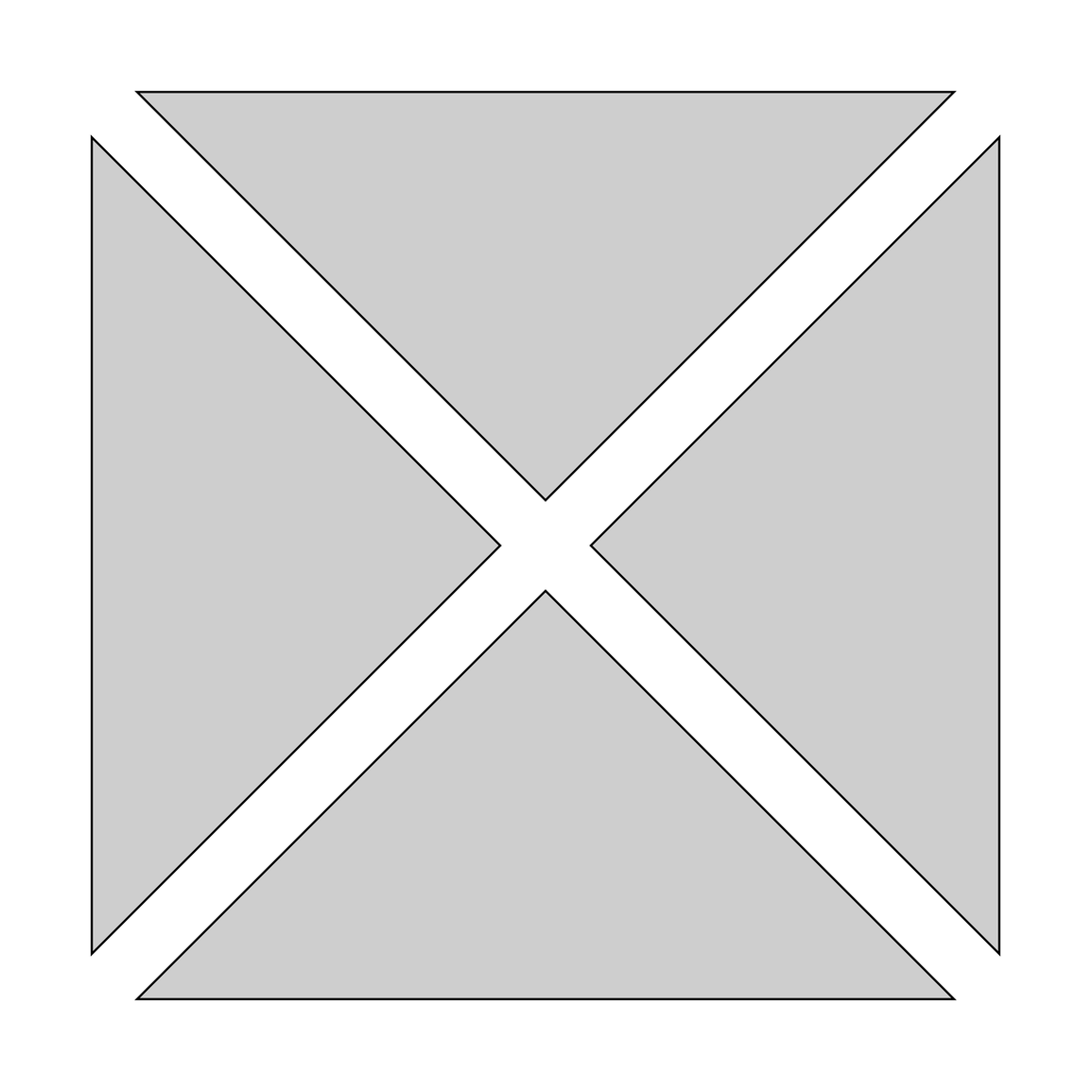}
$\overset{\overset{(3)}{\longrightarrow}}{\rule{0cm}{1.25cm}}$
\,
\includegraphics[height=2.75cm]{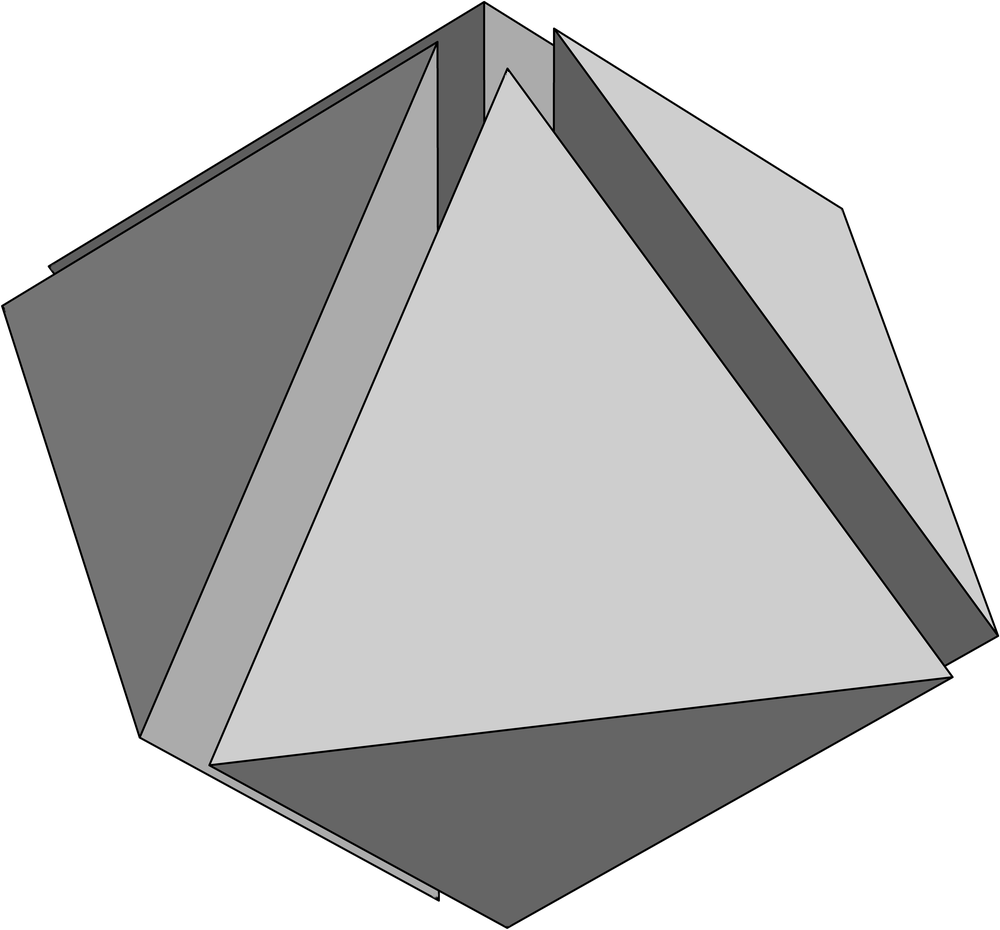}\\[\baselineskip]
\includegraphics[height=2.75cm]{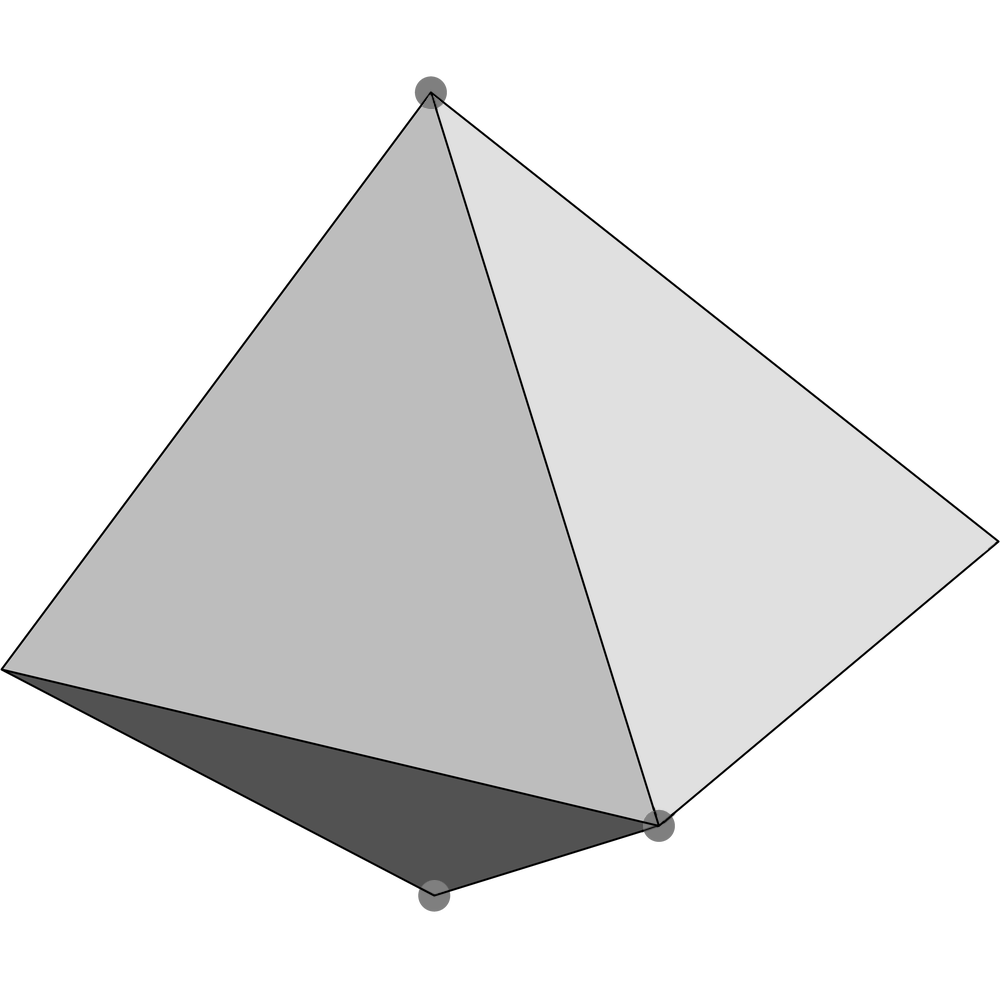}
\,
$\overset{\overset{(1)}{\longrightarrow}}{\rule{0cm}{1.25cm}}$
\includegraphics[height=2.75cm]{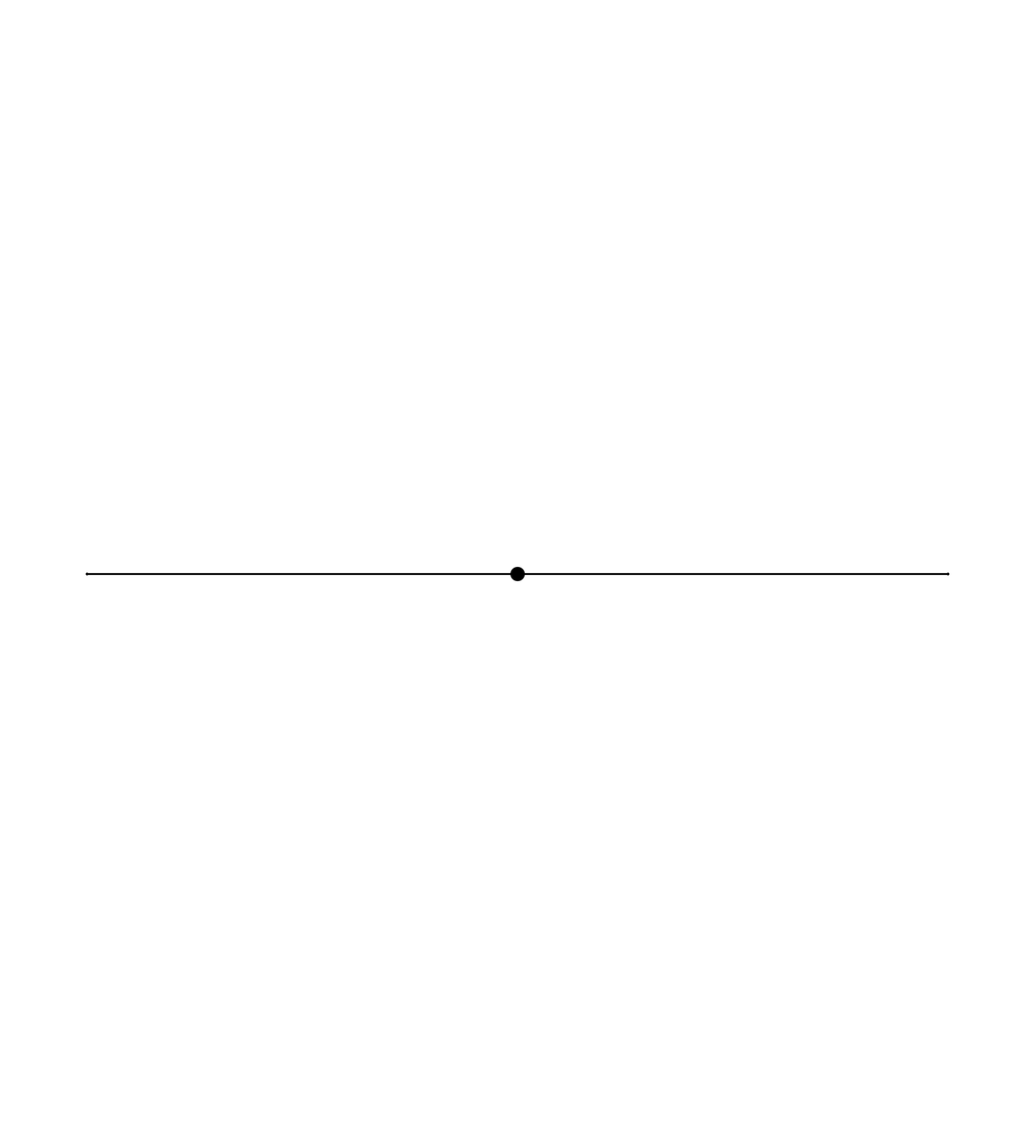}
$\overset{\overset{(2)}{\longrightarrow}}{\rule{0cm}{1.25cm}}$
\includegraphics[height=2.75cm]{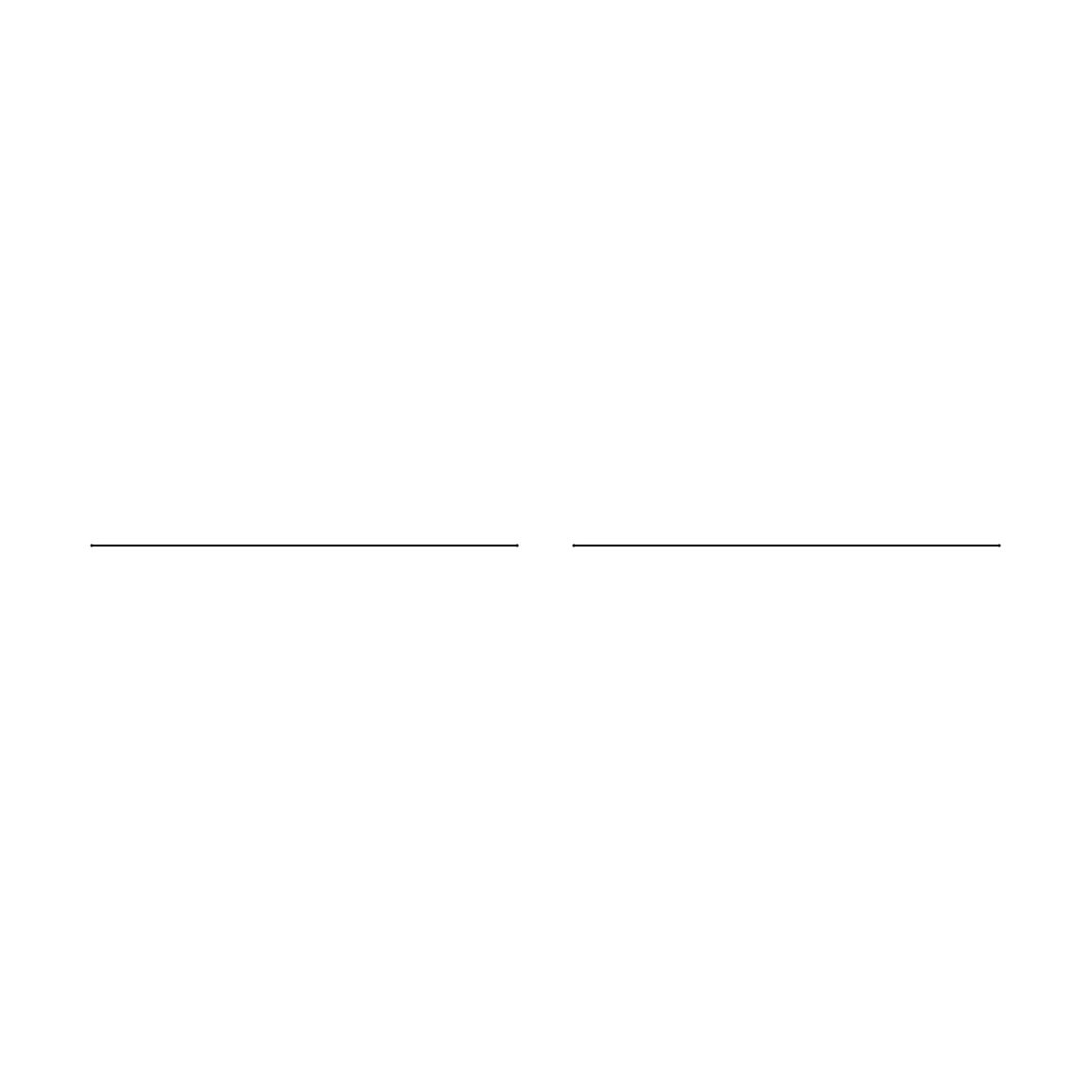}
$\overset{\overset{(3)}{\longrightarrow}}{\rule{0cm}{1.25cm}}$
\,
\includegraphics[height=2.75cm]{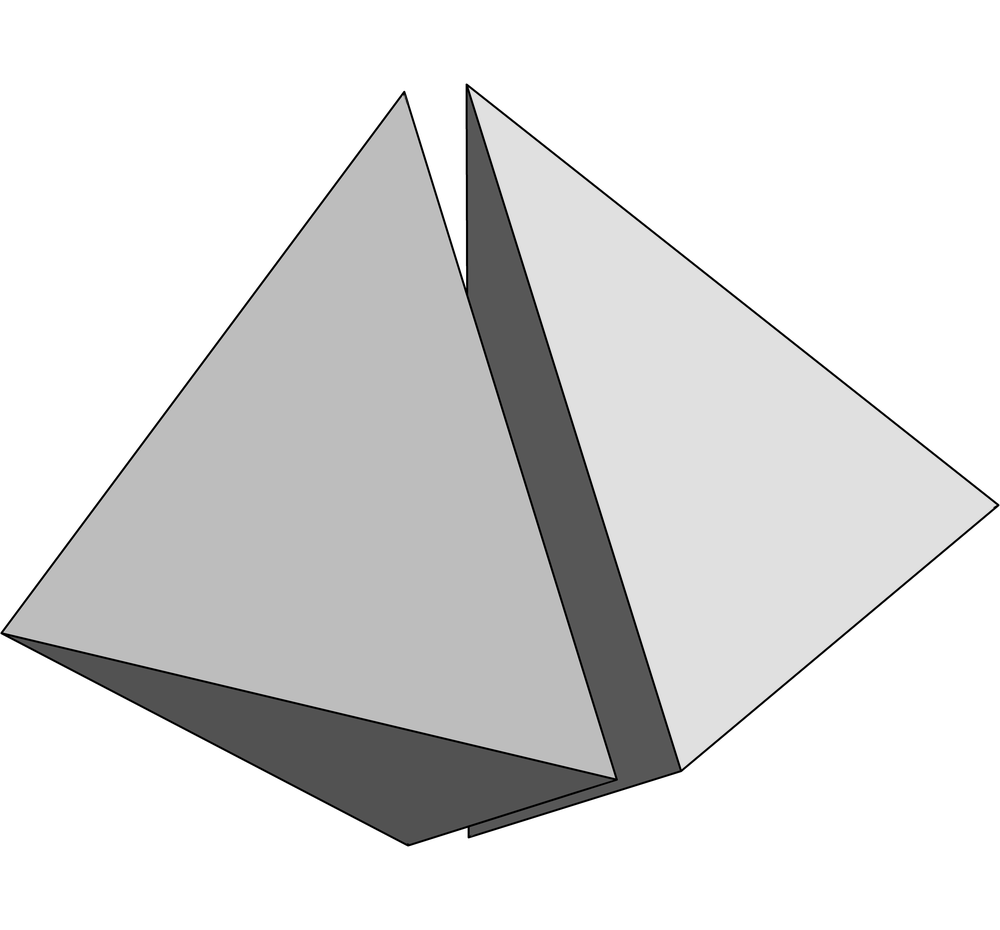}
\caption{Two examples of the lifting process: (1) Project polytope $\polytope$ to the orthogonal complement of the flat spanned by $\spineset$ (vertices marked by prominent dots) to obtain shadow $\fold{\polytope}$. (2) Star-triangulate $\fold{\polytope}$ with respect to the origin. (3) Lift star triangulation of $\fold{\polytope}$ to obtain $\spineset$-spinal triangulation of $\polytope$. Note: every facet of $\polytope$ contains exactly $\abs{\spineset}-1$ points of $\spineset$ in both examples.}
\label{fig:example_lifting_process}
\end{figure}
An important consequence of our characterization is that a spine
allows us to relate the volumes
of a convex polytope $\polytope$ and its shadow $\fold{\polytope}$
(with respect to that spine)
by a precise equation.
An application of this observation leads to our second result:
an exact volume formula of an important
polytope arising in number theory which we call the \emph{Everest polytope}.
We show that this polytope is the shadow of a higher-dimensional 
\emph{simplotope}, the product of simplices, whose volume is easy
to determine.

\paragraph{Number theoretic background.}
In the following we briefly discuss the number-theoretic background of the Everest polytope. G.~R. Everest \cite{Everest:1989,Everest:1990} studied various counting problems related to Diophantine equations. 
In particular, he proved asymptotic results for the number of values 
taken by a linear form whose variables are restricted to lie 
inside a given finitely generated subgroup of a number field. 
This includes norm form- and discriminant form equations, 
normal integral bases and related objects. 
Everest's work contains important contributions 
to the quantitative theory of $S$-unit equations 
and makes use of Baker's theory of linear forms in logarithms 
and Schmidt's subspace theorem from Diophantine approximation; 
see for instance \cite{Schlickewei:1990a,EvertseSchlickewei:2002,BakerWuestholz:1993}. 
Later, other authors \cite{FreiTichyZiegler:2014} applied the methods of Everest
to solve combinatorial problems in algebraic number fields. The corresponding counting results involve 
various important arithmetic constants, 
one of them being the volume of a certain convex polytope. 

In order to introduce Everest's constant, 
we use basic facts from algebraic number theory. 
Let $K$ be a number field, $N=N_{K\slash\mathbb{Q}}$ the field norm 
and $S$ a finite set of places of $K$ including the archimedian ones. 
We denote by $O_{K,S}=\{\alpha\in K:\lvert \alpha\rvert_{v}\leq 1\text{ for all }v\not\in S\}$ 
the ring of $S$-integers and its unit group by $U_{K,S}$; 
the group of $S$-units. 
Let $c_{0},\ldots , c_{n}$ denote given non-zero algebraic numbers. 
During the last decades, a lot of work is devoted
 to the study of the values taken by the expression 
$c_{0}x_{0}+\ldots +c_{n}x_{n}$, where the $x_{n}$ are allowed to run 
through $U_{K,S}$, see for instance \cite{Nishioka:1994,GyoryYu:2006}. A specific instance of this kind of general $S$-unit equations 
is the following combinatorial problem. 
As usual, two $S$-integers $\alpha$ and $\beta$ are said 
to be associated (for short $\alpha \sim \beta$) 
if there exists an $S$-unit $\epsilon$ such that $\alpha = \beta\epsilon$. 
It is well-known that the group of $S$-units $U_{K,S}$ is a free abelian group with $s=\lvert S\rvert-1$ generators, $\omega_{K}$ and $\operatorname{Reg}_{K,S}$ denote as usual the number of roots of unity and the $S$-regulator of $K$, respectively (for the basic concepts of algebraic number theory see \cite{Neukirch:1999}). 
Then for given $n\in \mathbb{N}, q>0$ the counting function $u(n,q)$ 
is defined as the number of equivalence classes $[\alpha]_{\sim}$ such that

$$N(\alpha)\ce \prod_{v\in S}\lvert\alpha\rvert _{v}\leq q,\; \alpha=\sum^{n}_{i=1}\varepsilon_{i},$$

\noindent
where $\varepsilon_{i}\in U_{K,S}$ and no subsum of $\varepsilon_{1}+\ldots +\varepsilon_{n}$ vanishes. From the work of Everest \cite{Everest:1989,Everest:1990}, the following asymptotic formula can be derived:
\begin{align*}
u(n,q)=\frac{c(n-1,s)}{n!}\left(\frac{\omega_{K}(\log q)^{s}}{\operatorname{Reg}_{K,S}}\right)^{n-1}+ o((\log q)^{(n-1)s-1+\varepsilon})
\end{align*}

\noindent
for arbitrary $\varepsilon>0$. Here $c(n-1,s)$ is a positive constant, and its exact value has been known only in special cases; see \cite{BarroeroFreiTichy:2011} for more details. In general, $c(n,s)$ is given as the volume of a convex polytope in~$\mathbb{R}^{ns}$
that we define and study in Section~\ref{sec:everest_polytope}.
Our results show that
\begin{align*}
c(n,s)= \frac{1}{(s!)^{n+1}}\frac{((n+1)s)!}{(ns)!},
\end{align*}
\noindent
which can also be written in terms of a multinomial coefficient as $\binom{(n+1)s}{s,\ldots,s}\frac{1}{(ns)!}$.

\paragraph{Geometric background.}
Our results fall into the category of \emph{constrained triangulations
of convex polytopes}. Triangulations of polytopes
are a classic topic in discrete geometry;
an infamous question is the quest for triangulating a $d$-hypercube with
a minimal number of simplices~\cite{CroftFalconerGuy:1991}. 
Precise answers are only known
up to dimension $7$ using computer-assisted proofs~\cite{HughesAnderson:1996}.
A contemporary overview on results relating to triangulations
of polytopes and more general point configurations is provided
by de~Loera, Rambau, and Santos~\cite{LoeraRambauSantos:2010}.
It includes a discussion on the triangulation of simplotopes, 
a geometric object whose study 
goes back to Hadwiger~\cite{Hadwiger:1957},
and has been studied, for instance, in the context 
of combinatorics~\cite{Freund:1986},
game theory~\cite{LaanTalman:1982}
and algebraic geometry~\cite[Ch.~7]{GelfandKapranovZelevinsky:2008}.
Simplotopes admit a standard triangulation,
the so-called \emph{staircase triangulation}, which can easily be described
in combinatorial terms. A by-product of our results is that simplotopes
can also be triangulated by a family of spinal triangulations.

An $\spinesetcard$-element subset $\spineset$ of the vertex set of a polytope 
with the property that each facet contains exactly $\spinesetcard-1$ points of $\spineset$
is called a \emph{special simplex} in the literature.
Special simplices
have been studied by Athanasiadis~\cite{Athanasiadis:2005}
to relate the Ehrhart polynomial of integer polytopes with special simplex
with the h-vector of the shadow. This results found applications 
in the study of toric rings and Gorenstein polytopes~\cite{HibiOhsugi:2006,BrunsRoemer:2007}. Polytopes with special simplices are further studied
by de Wolff~\cite{Wolff:UP}. His classification yields, among other results,
upper bounds for the number of faces of a polytope with special simplices.
Remarkably, special simplices with two vertices, called \emph{spindles},
are also used by Santos for his celebrated
counterexample for the Hirsch conjecture~\cite{Santos:2012}.
While these works employ similar techniques as our work, for instance,
lifting triangulations of the shadow to triangulations of the polytope,
the (more elementary) questions of this paper are not addressed in the related
work. We also point out that despite the close relation to special simplices,
the notion of spines introduced in this paper
is slightly more general because a facet is allowed to
contain all vertices of $\spineset$.

Otherwise, 
constraining triangulations has mostly been considered for low-dimensional
problems under an algorithmic angle. For instance, 
a \emph{constrained Delaunay triangulation} is a triangulation
which contains a fixed set of pre-determined simplices; apart from these
constraints, it tries to be ``as Delaunay as possible''; see
Shewchuk's work~\cite{Shewchuk:2008} for details. While our work is related
in spirit, there appears to be no direct connection to this framework,
because our constraint does not only ensure the presence of certain simplices
in the triangulation, but rather constrains all $d$-simplices at once.

Computing volumes of high-dimensional 
convex polytopes is another notoriously hard problem, 
from a computational perspective~\cite[Sec.~13]{Matousek:2002}\cite{EmirisFisikopoulos:2014} as well as
for special cases. A famous example is the \emph{Birkhoff polytope} of all doubly-stochastic $n\times n$-matrices,
whose exact volume is known exactly only up to $n=10$~\cite{Pak:2000,LoeraLiuYoshida:2009};
see Section~\ref{sec:birkhoff} for details.
Our contribution provides a novel technique to compute volumes
of polytopes through lifting into higher dimensions.
We point out that lifting increases the dimension, so that the lift
of a polytope is not the image of a linear transformation.
Therefore, the well-known formula
$\vol(A\polytope)=\sqrt{\det(A^TA)}\,\vol(\polytope)$
with $A\in\R^{e\times d}$ and $e\geq d$ does not apply to our case.
Another interpretation of our technique is to relate the volume
of a polytope and the volume of its (spinal) projection. 
For a fixed normal vector $\vector{u}$, the volume $\vol_\vector{u}$
of the projection of a polytope to the hyperplane normal to $\vector{u}$
is given by
\[\vol_{\vector{u}}=\frac{1}{2}\sum_{i=1}^m \abs{\vector{u}\tilde{\vector{u}}_i},\]
where $\tilde{\vector{u}}_1,\ldots,\tilde{\vector{u}}_m$ are the outward normal vectors
of the facets of the polytope
scaled by their volume~\cite[Thm~1.1]{BurgerGritzmannKlee:1996}.
The cited paper also discusses the algorithmic problem of finding
the projection vector that yields the maximal and minimal projected volume.
More results relating the volume of a polytope with its projections
(in codimension $1$) are studied 
under the name of \emph{geometric tomography}~\cite{Gardner:2006};
see also~\cite{Schneider:2014}.
We do not see a simple way
to derive the volume of a spinal projection using these approaches,
partially because our result relates the volume of polytopes 
of larger codimension.

\paragraph{Organization.}
We start by introducing the basic concepts from convex geometry 
in Section~\ref{sec:geometry}. 
We proceed with our structural result on spinal triangulations,
in Section~\ref{sec:characterization}.
We calculate the vertices of the Everest
polytope in Section~\ref{sec:everest_polytope} 
and define a map
from a simplotope to the Everest polytope 
in Section~\ref{sec:projections_of_simplotopes},
leading to the volume formula for the Everest polytope.
A case study of applying our results to the Birkhoff polytope
is given in Section~\ref{sec:birkhoff}.
We conclude with some additional remarks in Section~\ref{sec:conclusions_and_remarks}.

\smallskip

A conference version of this paper appeared at the 
International Symposium on Computational Geometry (SoCG) 2017~\cite{KerberTichyWeitzer:2017}.
This extended version contains the missing proof details that were omitted
in the conference version for brevity. Moreover, Section~\ref{sec:birkhoff}
did not appear in that previous version.


\section{Geometric concepts}
\label{sec:geometry}
Let $M$ be an arbitrary subset of $\R^d$ with some integer $d\geq 1$. The \emph{dimension} of $M$ is the dimension of the smallest affine subspace of $\R^d$ containing $M$.
We say that $M$ is \emph{full-dimensional}
if its dimension is equal to~$d$. Throughout the entire paper, $\pointset$ will always stand for a finite point set in $\R^d$ that is full-dimensional and in \emph{convex position}, that is $x\notin\conv(\pointset\setminus\{x\})$
for every $x\in\pointset$, where $\conv(\cdot)$ 
denotes the \emph{convex hull} in~$\R^d$.

\paragraph{Polytopes and simplicial complexes.}
We use the following standard definitions 
(compare, for instance, Ziegler~\cite{Ziegler:2007}):
A \emph{polytope} $\polytope$ is the convex hull of a finite point set in $\R^d$ in which case we say that the point set \emph{spans} $\polytope$.
A hyperplane $\hyperplane\subseteq\R^d$ is called \emph{supporting} (for $\polytope$) if $\polytope$ is contained in one of the closed half-spaces induced by $\hyperplane$.
A \emph{face} of $\polytope$ is either $\polytope$ itself, or the intersection of $\polytope$ with a supporting hyperplane.
If a face is neither the full polytope nor empty, we call it \emph{proper}.
A face of dimension $\ell$ is also called $\ell$-face of $\polytope$, with the convention that the empty set is a $(-1)$-face.
We call the union of all proper faces of $\polytope$ the \emph{boundary} of $\polytope$, and the points of $\polytope$ not on the boundary
the \emph{interior} of $\polytope$. 
$0$-faces are called the \emph{vertices} of $\polytope$, and we let $\SoVertices(\polytope)$ denote the set of vertices.
With $\ell$ being the dimension of $\polytope$, we call $(\ell-1)$-faces \emph{facets}, and $(\ell-2)$-faces \emph{ridges} of $\polytope$.
Any face $\face$ of $\polytope$ is itself a polytope whose vertex set is $\SoVertices(\polytope)\cap\face$.

It is well-known that every point $\vector{p}\in\polytope$ (and only those)
can be written as a \emph{convex combination} of vertices of $\polytope$, 
that is $\vector{p}=\sum_{\vector{v}\in\SoVertices(\polytope)} \lambda_{\vector{v}} \vector{v}$ 
with real values $\lambda_\vector{v}\geq 0$ for all $\vector{v}$ 
and $\sum_{\vector{v}\in\SoVertices(\polytope)}\lambda_\vector{v}=1$.
By Carath\'{e}odory's theorem, there exists a convex combination with at most $d+1$ non-zero entries, that is,
$\vector{p}=\sum_{i=1}^{d+1} \lambda_i \vector{v}_i$ with $\vector{v}_i\in\SoVertices(\polytope)$, $\lambda_i\geq 0$ and $\sum\lambda_i=1$.

\smallskip

An \emph{$\ell$-simplex} $\sigma$ with $\ell\in\set{-1,\ldots,d}$ is a polytope of dimension $\ell$ that has exactly $\ell+1$ vertices.
Every point in a simplex is determined by a unique convex combination of the vertices.
A \emph{simplicial complex} $\complex$ in $\R^d$ is a set of simplices in $\R^d$ 
such that for a simplex $\sigma$ in $\complex$, all faces of $\sigma$ are in $\complex$ as well, 
and if $\sigma$ and $\tau$ are in $\complex$, the intersection $\sigma\cap\tau$ is a common face of both (note that the empty set is a face of any polytope).
We let $\SoVertices(\complex)$ denote the set of all vertices in $\complex$.
The \emph{underlying space} $\underlyingspace{\complex}$ of $\complex$ is the union of its simplices.
We call a simplex in $\complex$ \emph{maximal} if it is not a proper face of another simplex in $\complex$.
A simplicial complex equals the set of its maximal simplices together with all their faces
and is therefore uniquely determined by its maximal simplices.
Also, the underlying space of $\complex$ equals the union of its maximal simplices. 

\smallskip

In what follows, we let $\polytope\ce \conv(\pointset)$ be the polytope spanned by $\pointset$ as fixed above.
In particular, $\dim(\polytope)=d$ because $\pointset$ is full-dimensional and $\SoVertices(\polytope)=\pointset$
because $\pointset$ is in convex position.

\paragraph{Spines.}
We call $\spineset\subseteq\pointset$ with $\abs{\spineset}=\spinesetcard$
a \emph{spine} of $\pointset$
if each facet of $\polytope$ contains at least $\spinesetcard-1$ points of $\spineset$.
Trivially, a one-point subset of $\pointset$ is a spine.
If $\polytope$ is a simplex, any non-empty subset
of vertices is a spine.
For a hypercube, every pair of opposite vertices forms a spine, 
but no other spines with two or more elements exist.

We derive an equivalent geometric characterization of spines next.
The \emph{$\spineset$-span (in $\pointset$)} 
is the set of all $d$-simplices $\sigma$
satisfying $\spineset\subseteq \SoVertices(\sigma)\subseteq\pointset$. Equivalently, it is the set of all $d$-simplices
with vertices in $\pointset$ which have $\conv(\spineset)$ as a common face. 
Clearly, each simplex $\sigma$ of the $\spineset$-span 
is contained in $\polytope$, 
and the same is true for the union of all simplices in the $\spineset$-span. 

\begin{lemma}
\label{lem:spine_facet_equiv}
Let $\spineset\subseteq\pointset$ with $\abs{\spineset}=\spinesetcard$. 
Then, $\spineset$ is a spine of $\pointset$ if and only 
if the union of all $\spineset$-span simplices is equal to $\polytope$, 
that is, 
if every point in $\polytope$ belongs to at least one simplex in the $\spineset$-span. 
\end{lemma}
\begin{proof}
We prove both directions of the equivalence separately. 
For ``$\Rightarrow$'', we proceed by induction on $\spinesetcard$. The statement is true for $\spinesetcard=0$ by Carath\'{e}odory's theorem. 
Let $\spineset$ be a set with at least one element, $\vector{u}\in\spineset$ arbitrary, and $\vector{p}\in\polytope\setminus\{\vector{u}\}$.
The ray starting in $\vector{u}$ through $\vector{p}$ leaves the polytope in a point $\bar{\vector{p}}$, 
and this point lies on (at least) one facet $\facet$ of $\polytope$ that does not contain $\vector{u}$. 
By assumption, $\facet$ contains all points in $\bar{\spineset}\ce\spineset\setminus\{\vector{u}\}$.
We claim that $\bar{\spineset}$ is a spine of $\bar{\pointset}\ce\pointset\cap\facet$.
To see that, note that $\facet$ is spanned by $\bar{\pointset}$
and the facets of $\facet$ (considered as a polytope in $\R^{d-1}$) are the ridges of $\polytope$ contained in $\facet$.
Every such ridge $\ridge$ is the intersection of $\facet$ and another facet $\facet'$ of $\polytope$. 
By assumption, $\facet'$ also contains at least $\spinesetcard-1$ points of $\spineset$
and it follows at once that $\ridge$ contains $\spinesetcard-2$ points of $\bar{\spineset}$.
So, $\bar{\spineset}$ is a spine of $\bar{\pointset}$,
and by induction hypothesis, there exists a $(d-1)$-simplex $\bar{\sigma}$ in the $\bar{\spineset}$-span in $\bar{\pointset}$ that contains $\bar{\vector{p}}$.
The vertices of $\bar{\sigma}$ together with $\vector{u}$ span a simplex $\sigma$ that contains $\vector{p}$
and $\sigma$ is in the $\spineset$-span by construction.

The direction ``$\Leftarrow$'' is clear if $\spinesetcard\in\set{0,1}$, so we may assume that $\spinesetcard\geq2$ and proceed by contraposition.
If $\spineset$ is not a spine, we have a facet $\facet$ of $\polytope$ 
such that less than $\spinesetcard-1$ points of $\spineset$ lie on $\facet$. 
Then, every simplex $\sigma$ in the $\spineset$-span has at least $2$ vertices not on $\facet$, and therefore
at most $d-1$ vertices on $\facet$. This implies that $\sigma\cap\facet$ is at most $(d-2)$-dimensional. 
Therefore, the (finite) union of all $\spineset$-span simplices cannot cover the $(d-1)$-dimensional facet $\facet$, which means that the $\spineset$-span
is not equal to $\polytope$.
\end{proof}

From now on, we 
use the (equivalent) geometric characterization from the preceding lemma and the combinatorial definition
of a spine interchangeably.
A useful property is that spines extend to faces in the following sense.

\begin{lemma}
\label{lem:spine_of_faces}
Let $\spineset$ be a spine of $\pointset$, and let $\face$ be an $\ell$-face of $\polytope$.
Then $\bar{\spineset}\ce\spineset\cap\face$ is a spine of $\bar{\pointset}\ce\pointset\cap\face$,
both considered as point sets in $\R^\ell$. In particular, $\face$ contains at least $\spinesetcard-(d-\ell)$ points of $\spineset$.
\end{lemma}
\begin{proof}
For every simplex $\sigma$ in the $\spineset$-span, let $\bar{\sigma}\ce\sigma\cap\face$.
Clearly, $\bar{\sigma}$ is itself a simplex, spanned by the vertices $\SoVertices(\bar{\sigma})=\SoVertices(\sigma)\cap\face$,
and is of dimension at most $\ell$.
Because $\spineset$ is a spine of $\pointset$, 
the union of all $\bar{\sigma}$ covers $\face=\conv(\bar{\pointset})$.
Moreover, $\SoVertices(\bar{\sigma})$ contains $\bar{\spineset}$.
If $\bar{\sigma}$ is not of dimension $\ell$, we can find an $\ell$-simplex in the $\bar{\spineset}$-span of $\bar{\pointset}$
that has $\bar{\sigma}$ as a face just by adding suitable vertices from $\bar{\pointset}$.
This implies that the union of the $\bar{\spineset}$-span covers $\face$. The ``in particular'' part follows by downward induction on $\ell$.
\end{proof}

\paragraph{Star triangulations.}
Let $\pointsetnonconvex\subseteq\R^d$ be a finite point set that is full-dimensional, but not necessarily in convex position.
We call a simplicial complex $\complex$ a \emph{triangulation} of $\pointsetnonconvex$ 
if $\SoVertices(\complex)=\pointsetnonconvex$ and $\underlyingspace{\complex}=\conv(\pointsetnonconvex)$. 
In this case, we also call $\complex$ a triangulation of the polytope $\conv(\pointsetnonconvex)$.
In a triangulation of $\pointsetnonconvex$, every maximal simplex must be of dimension $d$. 

We will consider several types of triangulations in this paper. For the first type, we assume that $\pointsetnonconvex=\pointset\cup\{\vector{0}\}$,
where $\vector{0}=(0,\ldots,0)\notin\pointset$, $\pointset$ is in convex position (as fixed before), and either $\vector{0}$ lies in $\conv(\pointset)$ 
or $\pointsetnonconvex$ is in convex position as well.
We define a \emph{star triangulation} of $\pointsetnonconvex$ as a triangulation where all $d$-simplices contain $\vector{0}$ as a vertex
(Figure~\ref{fig:star_triangulation}).

\begin{figure}[H]
\centering
\includegraphics[width=8cm]{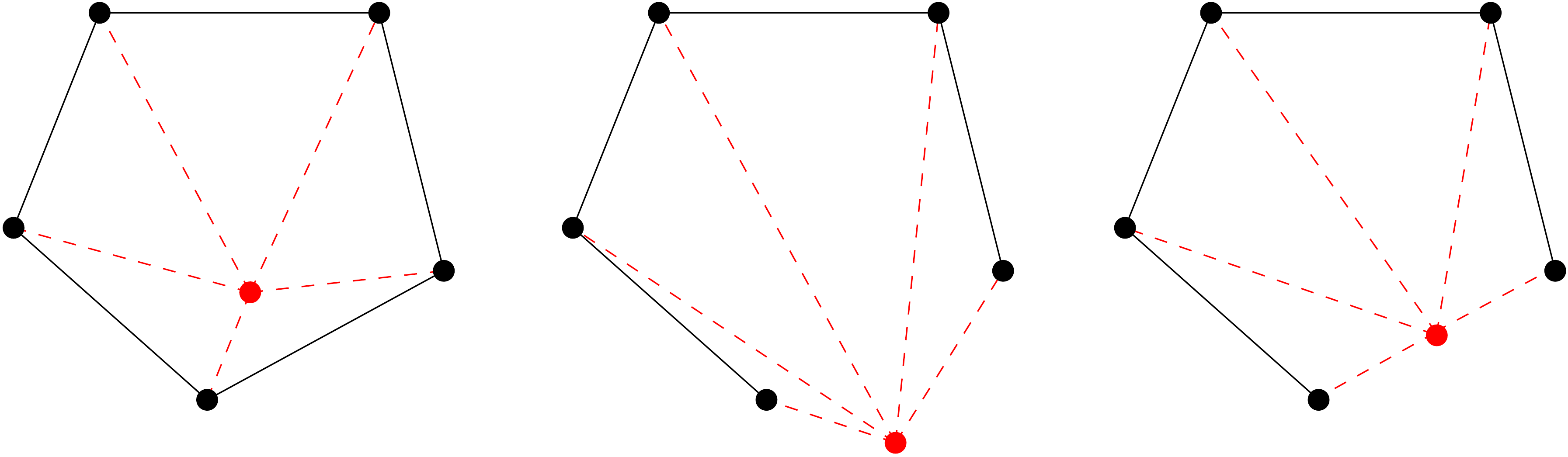}
\caption{The three types of star triangulations.}
\label{fig:star_triangulation}
\end{figure}

\begin{lemma}
\label{lem:star_exists}
A star triangulation of $\pointsetnonconvex$ exists.
\end{lemma}
\begin{proof}
We distinguish three cases, dependent on the position of $\vector{0}$. 
If it lies in the interior of $\conv(\pointset)$,
we consider an arbitrary triangulation of $\pointset$ (for instance, the Delaunay triangulation~\cite[Ch.~9]{BergCheongKreveldOvermars:2008}),
remove all interior simplices from this triangulation, and obtain a star triangulation by joining every $(d-1)$-simplex with $\vector{0}$.

If $\pointsetnonconvex$ is in convex position, we proceed similarly: Starting with an arbitrary triangulation of $\pointsetnonconvex$,
we remove the $d$-simplices to obtain a triangulation of the boundary facets, and join $\vector{0}$ with every $(d-1)$-simplex
that does not contain $\vector{0}$ to obtain a star triangulation.

It remains the case that $\vector{0}$ lies on the boundary of $\conv(\pointset)$. For this case, consider again a triangulation
of $\pointset$. Inductively in dimension, re-triangulate any $k$-face $\tau$ that contains $\vector{0}$ by joining $\vector{0}$
with each $(k-1)$-face on the boundary of $\tau$ that does not contain $\vector{0}$. The result of this process is a star triangulation.
\end{proof}
\noindent

\paragraph{Pulling triangulations.} 
As usual, let $\pointset$ be a point set in convex position spanning a polytope $\polytope$ in $\R^d$, 
and let $\vector{p}_1\in\pointset$.
We can describe a star triangulation with respect to $\vector{p}_1$ also as follows:
Triangulate each facet of $\polytope$ that does not contain $\vector{p}_1$
such that the triangulations agree on their common boundaries.
Writing $\Sigma\ce\{\sigma_1,\ldots,\sigma_m\}$ for the maximal simplices triangulating these facets, it is not difficult
to see that a (star) triangulation of $\polytope$ is given by the maximal simplices 
\[\vector{p}_1\ast\Sigma\ce\{\vector{p}_1\ast\sigma_1,\ldots,\vector{p}_1\ast\sigma_m\},\]
where $\vector{v}\ast\sigma$ is the simplex spanned by $\vector{v}$ and the vertices of $\sigma$. Recursively star-triangulating the facets
not containing $\vector{p}_1$ in the same manner, this construction gives rise to the \emph{pulling triangulation}.

To define the triangulation formally, we fix a total order $\vector{p}_1\prec\vector{p}_2\prec\ldots\prec\vector{p}_n$ on $\pointset$.
For a single point, we set $\pulling(\{\vector{p}\})\ce\{\vector{p}\}$.
For any face $\face$ of $\polytope$ with positive dimension, let $\vector{p}_k$ denote the smallest vertex of $\face$ with respect to $\prec$.
Then
\[\pulling(\face)\ce\vector{p}_k\ast \bigcup_{\stackrel{\text{$\ridge$ facet of $\face$}}{\vector{p}_k\notin\ridge}} \pulling(\ridge).\]
The following result follows directly by induction on the dimension of the faces. See~\cite[Sec.5.6]{BeckSanyal:2016}, \cite[Lemma 4.3.6]{LoeraRambauSantos:2010}:
\begin{theorem}
$\pulling(\polytope)$ are the maximal simplices of a triangulation of $\polytope$.
\end{theorem}

\section{Spinal triangulations}
\label{sec:characterization}

Fix a simplicial complex $\complex$ with vertex set $\SoVertices(\complex)$ in $\R^d$ 
and a set $\spineset\subseteq\SoVertices(\complex)$
of size at most $d+1$. Let $\sigma$ denote the simplex spanned by $\spineset$.
We call $\complex$ \emph{$\spineset$-spinal} if every maximal simplex of $\complex$ contains $\sigma$ as a face.
If $\complex$ is a triangulation of $\pointset$, 
we talk about a \emph{$\spineset$-spinal triangulation} accordingly.
$\spineset$-spinal triangulations are closely related to spines: 
if a $\spineset$-spinal triangulation of $\pointset$ exists, then $\spineset$ is a spine of $\pointset$,
because all maximal simplices of the triangulation lie in the $\spineset$-span.
For the previously discussed spine of a hypercube consisting of two opposite points, also a spinal triangulation
exists, consisting of $d!$ $d$-simplices that all share the diagonal connecting these points.
This construction is called \emph{staircase triangulation}~\cite{LoeraRambauSantos:2010} or \emph{Freudenthal triangulation}~\cite{EdelsbrunnerKerber:2012}.

We show next that a spine always induces a spinal triangulation.
Let $\polytope$ be a polytope spanned by a finite full-dimensional point set $\pointset\subseteq\R^d$ in convex position
and let $\spineset=\{\vector{u}_1,\ldots,\vector{u}_\spinesetcard\}$ be a spine of $\pointset$.
We fix a total order $\prec$ on $\pointset$ where $\vector{u}_1 \prec \vector{u}_2\prec\ldots\prec\vector{u}_\spinesetcard$ are the $\spinesetcard$ smallest elements,
preceeding all points in $\pointset\setminus\spineset$.

\begin{lemma}
\label{lem:spinal_triang_exists}
The pulling triangulation with respect to $\prec$ is a $\spineset$-spinal triangulation.
\end{lemma}
\begin{proof}
We prove the statement by induction on $\spinesetcard$. 
Let $\triangulation$ denote the pulling triangulation with respect to $\prec$.
For $\spinesetcard=1$, every maximal simplex of $\triangulation$ contains $\vector{u}_1$ by construction.
For $\spinesetcard>1$, every maximal simplex
is a join of $\vector{u}_1$ with a $(d-1)$-simplex $\sigma$ that is contained in a facet $\facet$ of $\polytope$ that
does not contain $\vector{u}_1$. $\sigma$, however, is itself a maximal simplex of the pulling triangulation
of $\facet$. By the spine property, $\spineset'\ce\{\vector{u}_2,\ldots,\vector{u}_\spinesetcard\}$ are vertices of $\facet$
and form a spine by Lemma~\ref{lem:spine_of_faces}. By induction, the pulling triangulation of $\facet$
is $\spineset'$-spinal. Therefore, $\sigma$ contains all vertices of $\spineset'$, and the join with $\vector{u}_1$
contains all vertices of $\spineset$.
\end{proof}

\paragraph{Folds and lifts.}
As before, let $\polytope$ be a polytope spanned by a finite full-dimensional point set $\pointset\subseteq\R^d$ in convex position, 
$\spineset\subseteq\pointset$ a spine of $\pointset$ with $\spinesetcard$ elements, and set $e\ce d-\spinesetcard+1$. 
Assume without loss of generality that the origin is among the points in $\spineset$. 
Furthermore let $\rotationaxis_{\spineset}$ be the subspace of $\R^d$ spanned by $\spineset$.
It is easy to see that the spine points are affinely independent, so that the dimension
of $\rotationaxis_{\spineset}$ is $\spinesetcard-1$.
Let $\rotationaxis_{\spineset}^\perp$ be the orthogonal complement, which is of dimension $e$.
Let $\linearmap_{\spineset}:\R^d\to\rotationaxis_{\spineset}^\perp$ the (orthogonal) projection of $\R^d$ to $\rotationaxis_{\spineset}^\perp$. 
For notational convenience, we use the short forms $\fold{x}\ce\linearmap_{\spineset}(x)$
and $\fold{X}\ce\linearmap_{\spineset}(X)$
for the images of points and sets in $\R^d$.

Fix a $\spineset$-spinal triangulation $\triangulation$ and let $\sigma$ be a maximal simplex of $\triangulation$.
Recall that the vertices of $\sigma$ are the points of $U$, which all map to $\vector{0}$ under $\linearmap_{\spineset}$,
and $e$ additional vertices $\vector{v}_1,\ldots,\vector{v}_e$. It follows that $\fold{\sigma}$ is 
the convex hull of $\{\vector{0},\fold{\vector{v}}_1,\ldots,\fold{\vector{v}}_e\}$. Moreover, since $\sigma$ has positive
($d$-dimensional) volume, its projection $\fold{\sigma}$ has positive ($e$-dimensional) volume as well.
It follows that $\fold{\sigma}$ is a $e$-simplex spanned by $\{\vector{0},\fold{\vector{v}}_1,\ldots,\fold{\vector{v}}_e\}$.
Consequently, with $\sigma_1,\ldots,\sigma_t$ being the maximal simplices of $\triangulation$,  
we call its \emph{fold} the set of simplices consisting of $\fold{\sigma}_1,\ldots,\fold{\sigma}_m$
and all their faces.
The following statement is a reformulation of~\cite[Prop.2.3]{Athanasiadis:2005} and~\cite[Prop.3.12]{ReinerWelker:2005}

\begin{lemma}
\label{lem:fold_is_star}
The fold of a $\spineset$-spinal triangulation $\triangulation$ is a star triangulation (with respect to the origin).
\end{lemma}
\begin{proof}
All maximal simplices of the fold contain the origin by construction. Moreover, their union covers $\fold{\polytope}$
because $\triangulation$ covers $\polytope$. Finally, we argue that the fold of two distinct maximal simplices $\sigma_1,\sigma_2$ cannot overlap:
Let $\hyperplane$ be a hyperplane that separates $\sigma_1$ and $\sigma_2$. Since $\rotationaxis_{\spineset}$ is contained in the affine span of $\sigma_1\cap\sigma_2$, $\hyperplane$ contains $\rotationaxis_{\spineset}$. Then, $\linearmap_{\spineset}(\hyperplane)$ is a hyperplane in $\R^e$
which separates the two $e$-simplices $\fold{\sigma_1}$ and $\fold{\sigma_2}$.
\end{proof}

We will now define the converse operation to get from a star triangulation in $\R^e$ to a $\spineset$-spinal triangulation in $\R^d$.
We first show that pre-images of vertices are well-defined.

\begin{lemma}
\label{lem:no_two_coincide}
If $\vector{v}\in\pointset\setminus\spineset$ then $\fold{\vector{v}}\neq\vector{0}$. Furthermore, if $\vector{v}\neq\vector{w}\in\pointset\setminus\spineset$ then $\fold{\vector{v}}\neq\fold{\vector{w}}$.
\end{lemma}

\begin{proof}
Since $\spineset$ is a spine of $\pointset$, there is a $d$-simplex $\sigma$ in the $\spineset$-span in $\pointset$ which has $\vector{v}$ among its vertices. If $\vector{v}$ is in the kernel of $\linearmap_{\spineset}$, then $\vector{v}$ and the points in $\spineset$ (which span the kernel of $\linearmap_{\spineset}$) are linearly dependent and $\sigma$ cannot be full-dimensional, which is a contradiction. Thus $\fold{\vector{v}}\neq\vector{0}$.

For the second part, we assume to the contrary that $\vector{v}\neq\vector{w}$ but $\fold{\vector{v}}=\fold{\vector{w}}$. If 
there is also a $d$-simplex $\sigma$ in the $\spineset$-span which has both $\vector{v}$ and $\vector{w}$ among its vertices, $\fold{\sigma}$ cannot be full-dimensional,
which is a contradiction. Otherwise, if no such $d$-simplex $\sigma$ exists, the $d$-simplices incident to $\vector{v}$ triangulate a neighborhood of $\vector{v}$ within $\polytope$,
and the same is true for $\vector{w}$, with the two sets of simplices being disjoint. It follows that their projections under $\linearmap_{\spineset}$ have to overlap,
contradicting Lemma~\ref{lem:fold_is_star}.
\end{proof}

For an $e$-simplex $\fold{\sigma}\subseteq\R^e$ with vertices in $\fold{\pointset}$ and containing $\vector{0}$ as vertex,
the \emph{lifted} $d$-simplex $\sigma\subseteq\R^d$ is spanned by the pre-image of $\SoVertices(\fold{\sigma})$ under $\linearmap_{\spineset}\big\vert_\pointset$ (the restriction of $\linearmap_{\spineset}$ to $\pointset$). 
Note the slight abuse of notation as we chose ``$\fold{\sigma}$'' as the name of a simplex before even defining the simplex $\sigma$, but the naming is justified because
$\fold{\sigma}$ indeed is equal to $\linearmap_{\spineset}(\sigma)$ in this case.
Given a star triangulation of $\fold{\pointset}$, its \emph{lift} is given by the set of lifts of its maximal simplices, together with all their faces.

Our goal is to show that the lift of a star triangulation is a $U$-spinal triangulation of $\polytope$. As a first step, we observe that such a lift is a $\spineset$-spinal 
simplicial complex.

\begin{lemma}
\label{lem:lift_is_simplicial}
The lift of a star triangulation of $\fold{\pointset}$ is a $\spineset$-spinal simplicial complex in $\R^d$ whose underlying space is a subset of $\polytope$.
\end{lemma}
\begin{proof}
Fix a star triangulation $\fold{\liftedstartriangulation}$ and let $\liftedstartriangulation$ denote its lift.
For any simplex in $\liftedstartriangulation$, all faces are included by construction.
We need to show that for two simplices $\sigma$ and $\tau$ in $\liftedstartriangulation$, 
$\sigma\cap\tau$ is a face of both.
We can assume without loss of generality that $\sigma$ and $\tau$ are maximal, hence $d$-simplices.
By construction, $\sigma$ and $\tau$ are the lifts of $e$-simplices $\fold{\sigma}$ and $\fold{\tau}$
in $\fold{\liftedstartriangulation}$.
Clearly, $\fold{\sigma}$ and $\fold{\tau}$ intersect because they share the vertex $\vector{0}$.
Moreover, since $\fold{\sigma}$ and $\fold{\tau}$ belong to a triangulation, their
intersection is a common face, 
spanned by a set of vertices $\{\vector{0},\fold{\vector{v}}_1\ldots,\fold{\vector{v}}_k\}$.
Hence, there exists a hyperplane $\fold{\hyperplane}$ in $\R^e$ separating $\fold{\sigma}$ and $\fold{\tau}$
such that $\fold{\sigma}\cap\fold{\hyperplane}=\conv\{\vector{0},\fold{\vector{v}}_1\ldots,\fold{\vector{v}}_k\}=\fold{\tau}\cap\fold{\hyperplane}$.
Let $\hyperplane$ denote the preimage of $\fold{\hyperplane}$ under $\linearmap_{\spineset}$.
Then, $\hyperplane$ is a separating hyperplane for $\sigma$ and $\tau$, 
and $\sigma\cap\hyperplane=\conv\{\vector{u}_1,\ldots,\vector{u}_{\spinesetcard},\vector{v}_1,\ldots,\vector{v}_k\}=\tau\cap\hyperplane$
as one can readily verify.
This shows that the lift is a simplicial complex. Its underlying space lies in $\polytope$ because every lifted simplex does. 
It is $\spineset$-spinal because the lift of every simplex contains $\spineset$ by definition.
\end{proof}

\paragraph{Volumes}
The converse of Lemma~\ref{lem:fold_is_star} follows from the fact that all lifts of star triangulations
have the same volume, as we will show next.

\begin{lemma}
\label{lem:volume_formula_simplices}
Let $\sigma$ be a simplex with vertices in $\pointset$ that contains $\spineset=\{\vector{u}_1,\ldots,\vector{u}_\spinesetcard\}$.
Then
\[ \binom{d}{n-1} \vol(\sigma) = \vol(\spineset) \vol(\fold{\sigma}),\]
where $\vol(\spineset)$ denotes the volume of the simplex spanned by $\spineset$.
\end{lemma}
\begin{proof}
For a $k$-simplex $\tau=\{v_0,\ldots,v_k\}$, let $\parallelotope{\tau}$ denote the parallelotope
spanned by $v_1-v_0,\ldots,v_k-v_0$. It is well-known that $\vol(\parallelotope{\tau})=k!\vol(\tau)$.
Rewriting the claimed volume by expanding the binomial coefficient and noting that $d-(n-1)=e$ yields
\[\underbrace{d! \vol(\sigma)}_{\vol(\parallelotope{\sigma})} = \underbrace{(n-1)! \vol(\spineset)}_{\vol(\parallelotope{\spineset})}\underbrace{e!\vol(\fold{\sigma})}_{\vol(\parallelotope{\fold{\sigma}})}.\]
To prove the relation between the volumes of paralleotopes, we assume without loss of generality that $\vector{u}_1=\vector{0}$.
Let $\rotationaxis_{\spineset}$ denote the linear subspace spanned by $u_2,\ldots,u_n$, and let
$\rotationaxis_{\spineset}^{\vector{q}}$ denote the parallel affine subspace that contains $\vector{q}\in\rotationaxis_{\spineset}^\perp$.
Then, $\vol(\parallelotope{\sigma}\cap \rotationaxis_{\spineset})=\vol(\parallelotope{\spineset})$ by definition. By Cavalieri's principle, every parallel cross-section 
of $\parallelotope{\sigma}$ has the same volume. More precisely,
\[\vol(\parallelotope{\sigma}\cap \rotationaxis_{\spineset}^{\vector{q}}) = \begin{cases}  \vol(\parallelotope{\spineset}) & \text{if $\vector{q}\in \parallelotope{\fold{\sigma}}$}\\ 0 & \text{otherwise}\\ \end{cases}.\]
Using Fubini's theorem, the volume of $\parallelotope{\sigma}$ can be expressed as an integral over all cross-sections, which yields
\[\vol(\parallelotope{\sigma})=\int_{\vector{q}\in\rotationaxis_{\spineset}^\perp}\vol(\parallelotope{\sigma}\cap \rotationaxis_{\spineset}^{\vector{q}})dq
=\int_{\vector{q}\in\parallelotope{\fold{\sigma}}} \vol(\parallelotope{\spineset})dq=\vol(\parallelotope{\spineset})\vol(\parallelotope{\fold{\sigma}}).\qedhere\]
\end{proof}

\begin{lemma}
\label{lem:volume_formula_lift}
Let $\spineset$ be a spine of $\pointset$, and
$\triangulation$ denote the lift of a star triangulation of $\fold{\pointset}$
Then, with the notation of Lemma~\ref{lem:volume_formula_simplices},
\[\binom{d}{n-1} \vol\left(\underlyingspace{\liftedstartriangulation}\right) = \vol(\spineset) \vol(\fold{\polytope}),\]
where $\underlyingspace{\liftedstartriangulation}$ is the underlying space of $\liftedstartriangulation$.
In particular, the underlying spaces of the lifts of all star triangulations of $\fold{\pointset}$ have the same volume.
\end{lemma}
\begin{proof}
The statement follows directly from Lemma~\ref{lem:volume_formula_simplices} because the relation holds for any simplex in the star triangulation and its lift.
\end{proof}

With that, we can prove our first main theorem.

\begin{maintheorem}[Lifting theorem]
\label{thm:lifting_theorem}
Let $\polytope$ be a polytope spanned by a full-dimensional finite point set $\pointset\subseteq\R^d$ in convex position and $\spineset\ce\{\vector{u}_1,\ldots,\vector{u}_\spinesetcard\}\subseteq\pointset$. Then, 
\[
\text{$\spineset$ is a spine of $\pointset$ if and only if there exists
a $\spineset$-spinal triangulation of $\pointset$.}\]
In this case, for $\spinesetcard\geq1$, the $\spineset$-spinal triangulations of $\pointset$ are exactly the lifts of the star triangulations of $\fold{\pointset}$, the 
orthogonal projection of $\pointset$ to the orthogonal complement of the affine space
spanned by $\spineset$.
Furthermore, if $\spinesetcard\geq 2$, 
\begin{align*}
\binom{d}{n-1}\vol(\polytope)=\vol(U)\vol(\fold{\polytope}).
\end{align*}
\end{maintheorem}

\begin{proof}
For the equivalence, 
the ``if''-part follows directly
from the geometric characterization
of spines (Lemma~\ref{lem:spine_facet_equiv}), and the ``only if'' part
follows from Lemma~\ref{lem:spinal_triang_exists}.

For the second part, given any $\spineset$-spinal triangulation $\triangulation^\ast$, 
its fold $\fold{\triangulation}^\ast$ is a star triangulation 
by Lemma~\ref{lem:fold_is_star}, and by lifting that star triangulation, we obtain
$\triangulation^\ast$ back. Vice versa, starting with any star triangulation
$\fold{\triangulation}$, we know that its lift is a simplicial complex contained in $\polytope$
by Lemma~\ref{lem:lift_is_simplicial}.
By Lemma~\ref{lem:volume_formula_lift}, the lifts
of $\fold{\triangulation}$ and $\fold{\triangulation}^\ast$ have the same volume, but the lift of the latter
is $\polytope$. It follows that also the lift of the former is a triangulation of $\polytope$,
proving the second claim.

The claim about the volumes follows at once by applying Lemma~\ref{lem:volume_formula_lift}
to $\triangulation^\ast$ and $\fold{\triangulation}^\ast$.
\end{proof}

We remark that not every $\spineset$-spinal triangulation is a pulling triangulation. This follows
from the fact that each pulling triangulation is \emph{regular} (see~\cite[p.181]{LoeraRambauSantos:2010}),
but examples of non-regular spinal triangulations are known (one such example
is given in~\cite[p.306]{LoeraRambauSantos:2010}).


\section{The Everest polytope}
\label{sec:everest_polytope}

For $n,s\in\N$, define $\everestpolytopelong{n}{s}\ce\set{\vector{x}\in\R^{ns}\mid g_{n,s}(\vector{x})\leq1}$ where $g_{n,s}:\R^{ns}\to\R$ with
\begin{align*}
(x_{1,1},\ldots,x_{n,s})&\mapsto\sum_{j=1}^s\max\set{0,x_{1,j},\ldots,x_{n,j}}+\max\set{0,-\sum_{j=1}^sx_{1,j},\ldots,-\sum_{j=1}^sx_{n,j}}.
\end{align*}
It is not difficult to verify that 
$\everestpolytopelong{n}{s}$ is bounded and  
the intersection of finitely many halfspaces of $\R^{ns}$.
We call it the \emph{$(n,s)$-Everest polytope}.
It is well-known \cite{BarroeroFreiTichy:2011} that the number-theoretic constant
$c(n,s)$ discussed in the introduction is equal 
to the volume of $\everestpolytopelong{n}{s}$.

\paragraph{Vertex sets.}
In order to describe the vertices of $\everestpolytopelong{n}{s}$ we introduce the following point sets which we also utilize in later parts of the paper. 
Note that we identify $\R^{ns}$ and $\R^{n\times s}$ which explains the meaning of ``row'' and ``column'' in the definition.
Let $\standard_s(i)$ denote the $i$-th $s$-dimensional unit (row) vector 
with the convention that $\standard_s(0)=\vector{0}$. 
We define the following sets of points in $\R^{ns}$:
\begin{align*}
\verticesminusone{n}{s}&\ce
\left\{
\left(
\begin{matrix}
-\standard_s(j_1)\\
\vdots\\
-\standard_s(j_n)
\end{matrix}
\right)
\;\,
\right|
\;
\left.
j_1,\ldots,j_n\in\set{0,\ldots,s}
\vphantom{
\left(
\begin{matrix}
-\standard_s(j_1)\\
\vdots\\
-\standard_s(j_n)
\end{matrix}
\right)
}
\right\},\\
\verticeszero{n}{s}&\ce
\left\{
\left(
\begin{matrix}
-\standard_s(j)\\
\vdots\\
-\standard_s(j)
\end{matrix}
\right)
\;\,
\right|
\;
\left.
j\in\set{0,\ldots,s}
\vphantom{
\left(
\begin{matrix}
-\standard_s(j)\\
\vdots\\
-\standard_s(j)
\end{matrix}
\right)
}
\right\},\\
\verticesone{n}{s}&\ce\verticesminusone{n}{s}-(\verticeszero{n}{s}\setminus\set{\vector{0}})=\set{\vector{v}-\vector{u}\mid\vector{v}\in\verticesminusone{n}{s}\land\vector{u}\in\verticeszero{n}{s}\setminus\set{\vector{0}}}.
\end{align*}

\noindent
It can be readily verified that $\verticesminusone{n}{s}$ is the set of points in $\set{-1,0}^{ns}$ such that there is at most one $-1$ per row, $\verticeszero{n}{s}$ is the set of points in $\verticesminusone{n}{s}$ such that all $-1$'s (if there are any) are contained in a single column, and $\verticesone{n}{s}$ is the set of points in $\set{-1,0,1}^{ns}$ such that\\[0.5\baselineskip]
$\vphantom{a}\quad\bullet$ all '$1$'s (if there are any) are in a unique ''$1$-column``,\\
$\vphantom{a}\quad\bullet$ all entries of the $1$-column are either $0$ or $1$,\\
$\vphantom{a}\quad\bullet$ all rows with a $1$ contain at most one $-1$,\\
$\vphantom{a}\quad\bullet$ all rows without a $1$ contain only '$0$'s.

\begin{lemma}
\label{lem:vertices_card}
$\verticesone{n}{s}\cap\verticesminusone{n}{s}=\set{\vector{0}}$, $\verticeszero{n}{s}\subseteq\verticesminusone{n}{s}$, $\abs{\verticesminusone{n}{s}}=(s+1)^n$, $\abs{\verticeszero{n}{s}}=s+1$, and $\abs{\verticesone{n}{s}}=s(s+1)^n-s+1$.
\end{lemma}

\begin{proof}
Follows directly from the definitions and from basic combinatorics.
\end{proof}

\begin{theorem}
\label{thm:everest_vertices}
The set of vertices of $\everestpolytopelong{n}{s}$ is given by $(\verticesone{n}{s}\cup\verticesminusone{n}{s})\setminus\set{\vector{0}}=(\verticesminusone{n}{s}-\verticeszero{n}{s})\setminus\set{\vector{0}}$.
\end{theorem}

\noindent
We split the proof into several parts which will be treated in the following lemmas. For the rest of this section, let $i$ and $j$ (with a possible subscript) denote elements of $\set{1,\ldots,n}$ and $\set{1,\ldots,s}$, respectively, let $\vector{v}=(v_{1,1},\ldots,v_{n,s})$ be a vertex of $\everestpolytopelong{n}{s}$, and set
\begin{align*}
m_j&\ce\max\set{0,v_{1,j},\ldots,v_{n,j}}\text{ for all $j$},\\
s_i&\ce-\sum_{j=1}^sv_{i,j}\text{ for all $i$},\\
m&\ce\max\set{0,s_1,\ldots,s_n}.
\end{align*}
Then
\begin{align*}
g_{n,s}\left(
\begin{matrix}
v_{1,1}&\cdots&v_{1,s}\\
\vdots&&\vdots\\
v_{n,1}&\cdots&v_{n,s}
\end{matrix}
\right)
&=
\max
\left\{
\begin{matrix}
0\\
v_{1,1}\\
\vdots\\
v_{n,1}
\end{matrix}
\right\}
+\cdots+
\max
\left\{
\begin{matrix}
0\\
v_{1,s}\\
\vdots\\
v_{n,s}
\end{matrix}
\right\}
+
\max
\left\{
\begin{matrix}
0\\
-v_{1,1}-\cdots-v_{1,s}\\
\vdots\\
-v_{n,1}-\cdots-v_{n,s}
\end{matrix}
\right\}\\
&=m_1+\cdots+m_s+m\\
&=1.
\end{align*}
Furthermore it can easily be verified that $v_{i,j}\in[-1,1]$ for all $i$ and $j$, $m_j\in[0,1]$ for all $j$, $s_i\in[-1,1]$ for all $i$, and $m\in[0,1]$.

In the proofs below, we will repeatedly apply the following argument: if there is an $\varepsilon>0$ and an $\vector{x}\in\R^{ns}$ such that $\vector{v}\pm\delta\vector{x}\in\everestpolytopelong{n}{s}$ for all $\delta\in[0,\varepsilon]$, $\vector{v}$ cannot be a vertex of $\everestpolytopelong{n}{s}$ (since it is in the interior of an at least $1$-dimensional face).

\begin{lemma}
\label{lem:everest_vertices_minusone}
If $m_j=0$ for all $j$ then $\vector{v}\in\verticesminusone{n}{s}$.
\end{lemma}

\begin{proof}
Since all $m_j$ are equal to zero, all $v_{i,j}$ have to be non-positive, so all $s_i$ are non-negative. Also we get that $m$ is equal to $1$ which implies that at least one of the $s_i$ is equal to $1$. Suppose that $v_{i_0,j_0}\in(-1,0)$ for some $i_0,j_0$. If $s_{i_0}<1$, then for $\varepsilon\ce\min\set{-v_{i_0,j_0},1-s_{i_0}}>0$ and $\delta\in[0,\varepsilon]$ we get that $g_{n,s}(v_{1,1},\ldots,v_{i_0,j_0}\pm\delta,\ldots,v_{n,s})=1$, so $(v_{1,1},\ldots,v_{i_0,j_0}\pm\delta,\ldots,v_{n,s})$ is on the boundary of $\everestpolytopelong{n}{s}$ and $\vector{v}$ cannot be a vertex of $\everestpolytopelong{n}{s}$. To see this more easily consider the example
\begin{flalign*}
\quad\vector{v}=\left(\begin{matrix}
-1&0&0&0\\
0&-1&0&0\\
-1/6&0&-2/3&0
\end{matrix}\right)
\longrightarrow
\left(\begin{matrix}
-1&0&0&0\\
0&-1&0&0\\
-1/6\pm\delta&0&-2/3&0
\end{matrix}\right).&&
\end{flalign*}
If on the other hand $s_{i_0}$ is equal to $1$, then there is a $j_1\neq j_0$ such that $v_{i_0,j_1}\in(-1,0)$. But then we get $g_{n,s}(v_{1,1},\ldots,v_{i_0,j_0}\pm\delta,\ldots,v_{i_0,j_1}\mp\delta,\ldots,v_{n,s})=1$ for $\varepsilon\ce\min\set{-v_{i_0,j_0},-v_{i_0,j_1},v_{i_0,j_0}+1,v_{i_0,j_1}+1}>0$ and $\delta\in[0,\varepsilon]$, so again $\vector{v}$ cannot be a vertex of $\everestpolytopelong{n}{s}$. Again we give an example
\begin{flalign*}
\quad\vector{v}=\left(\begin{matrix}
-1&0&0&0\\
0&-1&0&0\\
-1/6&-1/6&0&-2/3
\end{matrix}\right)
\longrightarrow
\left(\begin{matrix}
-1&0&0&0\\
0&-1&0&0\\
-1/6\pm\delta&-1/6\mp\delta&0&-2/3
\end{matrix}\right).&&
\end{flalign*}
Thus we get that all $v_{i,j}$ are either $-1$ or $0$ and it is clear that in any given row $i_0$ only one of the $v_{i_0,j}$ can be $-1$ (as they sum up to $-s_{i_0}\geq-1$). Therefore $\vector{v}\in\verticesminusone{n}{s}$.
\end{proof}

\begin{lemma}
\label{lem:everest_vertices_one}
If $m_{j_0}=1$ for some $j_0$ then $\vector{v}\in\verticesone{n}{s}$.
\end{lemma}

\begin{proof}
Since $m_{j_0}$ is equal to $1$, all other $m_j$ and $m$ have to be equal to zero. Thus all $v_{i,j_0}$ are non-negative and at least one of them is equal to $1$. Also, all other $v_{i,j}$ (i.e. if $j\neq j_0$) are non-positive and so are all $s_i$. Just as in the proof of Lemma~\ref{lem:everest_vertices_minusone} we can show that all $v_{i,j}$ are either $-1$, $0$, or $1$; we omit the details. Furthermore it is clear that if $v_{i_0,j_0}$ is equal to $1$ for some $i_0$, then there cannot be more than one $-1$ in the $i_0$-th row (as $s_{i_0}\leq0$). By the same reasoning, if $v_{i_0,j_0}$ is equal to $0$, there cannot be any $-1$ in the $i_0$-th row at all. Considering the definition of $\verticesone{n}{s}$ we thus see that $\vector{v}\in\verticesone{n}{s}$.
\end{proof}

\begin{lemma}
\label{lem:everest_vertices_all}
If $m_j\neq1$ for all $j$ then $m_j=0$ for all $j$.
\end{lemma}

\begin{proof}
We assume to the contrary that $m_{j_0}\in(0,1)$ for some $j_0$. Suppose that all other $m_j$ are equal to zero. Then $m=1-m_{j_0}\in(0,1)$ and for $\varepsilon\ce\min\set{m_{j_0},1-m_{j_0}}>0$ and $\delta\in[0,\varepsilon]$ we get that $g_{n,s}(v_{1,1},\ldots,v_{1,j_0}\pm\delta,\ldots,v_{1,s},\ldots,v_{n,1},\ldots,v_{n,j_0}\pm\delta,\ldots,v_{n,s})=1$, hence $\vector{v}$ is not a vertex of $\everestpolytopelong{n}{s}$. As an example consider
\begin{flalign*}
\quad\vector{v}=\left(\begin{matrix}
0&1/3&0&-2/3\\
0&-1/2&0&-1/4\\
0&1/6&-1/6&-2/3
\end{matrix}\right)
\longrightarrow
\left(\begin{matrix}
0&1/3\pm\delta&0&-2/3\\
0&-1/2\pm\delta&0&-1/4\\
0&1/6\pm\delta&-1/6&-2/3
\end{matrix}\right).&&
\end{flalign*}
If on the other hand there is another $m_{j_1}$ that is not equal to zero we get $g_{n,s}(x_{1,1},\ldots,x_{1,j_0}\pm\delta,\ldots,x_{1,j_1}\mp\delta,\ldots,x_{1,s},\ldots,x_{n,1},\ldots,x_{n,j_0}\pm\delta,\ldots,x_{n,j_1}\mp\delta,\ldots,x_{n,s})=1$ for $\varepsilon\ce\min\set{m_{j_0},m_{j_1},1-m_{j_0},1-m_{j_1}}>0$ and $\delta\in[0,\varepsilon]$, so again $\vector{v}$ is not a vertex of $\everestpolytopelong{n}{s}$. An example for this situation is
\begin{flalign*}
\quad\vector{v}=\left(\begin{matrix}
0&1/3&0&-1/6\\
0&-1/2&1/2&0\\
0&1/6&2/3&-2/3
\end{matrix}\right)
\longrightarrow
\left(\begin{matrix}
0&1/3\pm\delta&0\mp\delta&-1/6\\
0&-1/2\pm\delta&1/2\mp\delta&0\\
0&1/6\pm\delta&2/3\mp\delta&-2/3
\end{matrix}\right).&&
\end{flalign*}
\end{proof}

\begin{proof}[Proof of Theorem~\ref{thm:everest_vertices}]
Lemma~\ref{lem:everest_vertices_all} implies that if $\vector{v}$ is a vertex of $\everestpolytopelong{n}{s}$, then we are in the situation of either Lemma~\ref{lem:everest_vertices_minusone} or Lemma~\ref{lem:everest_vertices_one}, hence $\vector{v}\in\verticesone{n}{s}\cup\verticesminusone{n}{s}$. Furthermore it is clear that $\vector{0}$ is not a vertex of $\everestpolytopelong{n}{s}$. Also, it follows from the definition of $\verticesone{n}{s}$ that
\begin{align*}
\verticeseverest{n}{s}\ce(\verticesone{n}{s}\cup\verticesminusone{n}{s})\setminus\set{\vector{0}}&=
((\verticesminusone{n}{s}-(\verticeszero{n}{s}\setminus\set{\vector{0}}))\cup\verticesminusone{n}{s})\setminus\set{\vector{0}}=
(\verticesminusone{n}{s}-\verticeszero{n}{s})\setminus\set{\vector{0}}.
\end{align*}

We are left to show that $\verticeseverest{n}{s}\subseteq\SoVertices(\everestpolytopelong{n}{s})$. First we observe that $g_{n,s}(\vector{v})=1$ for all $\vector{v}\in\verticeseverest{n}{s}$. We consider the case that $n,s\geq2$ and assume that there is a $\vector{v}\in\verticeseverest{n}{s}$ that is not a vertex of $\everestpolytopelong{n}{s}$. Since $\vector{v}$ is on the boundary of $\everestpolytopelong{n}{s}$ but not a vertex of $\everestpolytopelong{n}{s}$, it is contained in the interior of an at least $1$-dimensional face $\face$ of $\everestpolytopelong{n}{s}$. Let $\vector{w}$ be any vertex of $\face$. Then $\vector{w}\in\verticeseverest{n}{s}$ and $\vector{v}\neq\vector{w}$.

Now let $\vector{a}\in\verticeseverest{2}{2}$, $\vector{a}\neq\vector{b}\in\verticeseverest{2}{2}\cup\set{\vector{0}}$, and consider the convex combination $\alpha\vector{a}+(1-\alpha)\vector{b}$, $\alpha\in\R$. By plugging in all possible values of $\vector{a}$ and $\vector{b}$ one can verify that if $\alpha>1$ then $g_{2,2}(\alpha\vector{a}+(1-\alpha)\vector{b})>1$.

Let $\vector{v}',\vector{w}'\in\R^{2\times s}$ be submatrices consisting of $2$ rows of $\vector{v}$ and $\vector{w}$ respectively, such that $\vector{v}'\neq\vector{0}$ and $\vector{v}'\neq\vector{w}'$. By definition of $\verticeseverest{n}{s}$, $\vector{v}'$ and $\vector{w}'$ thus respectively contain submatrices of the form $\vector{a}$ and $\vector{b}$ from above while the remaining entries are padded with zeros. It follows from the definition of $g_{n,s}$ that $g_{n,s}(\alpha\vector{v}+(1-\alpha)\vector{w})>1$ if $\alpha>1$, which contradicts the fact that $\vector{v}$ is in the interior of $\face$. Hence, $\verticeseverest{n}{s}\subseteq\SoVertices(\everestpolytopelong{n}{s})$ if $n,s\geq2$. If $n=1$ and $s\geq2$ one can proceed analogously by considering $\vector{a}\in\verticeseverest{1}{2}$, $\vector{a}\neq\vector{b}\in\verticeseverest{1}{2}\cup\set{\vector{0}}$; same goes for $s=1$ and $\vector{a}\in\verticeseverest{1}{1}$, $\vector{a}\neq\vector{b}\in\verticeseverest{1}{1}\cup\set{\vector{0}}$.
\end{proof}

\begin{corollary}
The number of vertices of $\everestpolytopelong{n}{s}$ is given by $(s+1)^{n+1}-s-1$.
\end{corollary}

\begin{proof}
Follows directly from Theorem~\ref{thm:everest_vertices} and Lemma~\ref{lem:vertices_card}.
\end{proof}


\section{Projections of simplotopes}
\label{sec:projections_of_simplotopes}

We will establish a relation between the Everest polytope $\everestpolytopelong{n}{s}$ and a special polytope known as simplotope.
This relation will allow the comparison of the volumes of the two polytopes 
even though they are of different dimension. 

\paragraph{Simplotopes.}
For $s\in\N$, the $s$-simplex $\simplex_s$ is spanned by the points $(\vector{0},-\standard_s(1),\ldots,-\standard_s(s))$ in $\R^{s}$, 
with $\standard_s(i)$ the $i$-th standard vector in $\R^s$, as before.
A \emph{simplotope} is a Cartesian product of the form $\simplex_{s_1}\times\ldots\times\simplex_{s_n}$ with positive integers $s_1,\ldots,s_n$. 
Note that in the literature, simplotopes are usually defined in a combinatorially equivalent way using the standard $s$-simplex 
spanned by $(s+1)$-unit vectors in $\R^{s+1}$. 
We restrict to the case that all $s_i$ are equal, 
and we call \[\simplotope_{n,s}=\underbrace{\simplex_s\times\ldots\times\simplex_s}_{\text{$n$ times}}\] the \emph{$(n,s)$-simplotope} for $n,s\in\N$.

For instance, $d$-hypercubes are $d$-fold products of line segments, and therefore $(n,s)$-simplotopes with $n=d$ and $s=1$.
It is instructive to visualize a point in $\simplotope_{n,s}$ as an $n\times s$-matrix with real entries in $[0,1]$, 
where the sums of the entries in each row do not exceed $1$. One can readily verify that the set of vertices of the $(n,s)$-simplotope is equal to $\verticesminusone{n}{s}$, as given in the beginning of Section~\ref{sec:everest_polytope}.
Moreover, it is straight-forward to verify that each facet of $\simplotope_{n,s}$ is given by
\[\underbrace{\simplex_s\times\ldots\times\simplex_s}_{\text{$i$ times}}\times \facet\times \underbrace{\simplex_s\times\ldots\times\simplex_s}_{\text{$n-i-1$ times}}\]
where $\facet$ is a facet of $\simplex_s$ and $0\leq i\leq n-1$. It follows at once:

\begin{theorem}
\label{thm:simplotope_spine}
$\verticeszero{n}{s}$ is a spine of $\verticesminusone{n}{s}=\SoVertices(\simplotope_{n,s})$.
\end{theorem}

\paragraph{A linear transformation.}
We call the matrix
\begin{align*}
\setransform_{n,s}\ce
\left(\begin{matrix}
&-I_s\\
I_{ns}&\vdots\\
&-I_s
\end{matrix}\right)
=
\left(\begin{matrix}
1&&&&&&0&-1&&0\\
&\ddots&&&&&&&\ddots&\\
&&\ddots&&&&&0&&-1\\
&&&\ddots&&&&&\vdots&\\
&&&&\ddots&&&-1&&0\\
&&&&&\ddots&&&\ddots&\\
0&&&&&&1&0&&-1
\end{matrix}\right)\in\R^{(ns)\times((n+1)s)},
\end{align*}
where $I_d$ is the identity matrix of dimension $d$, the \emph{$(n,s)$-SE-transformation} (``SE'' stands for ``\textbf{S}implotope $\leftrightarrow$ \textbf{E}verest polytope''). We show that the name is justified, as it maps the $(n+1,s)$-simplotope onto the $(n,s)$-Everest polytope.

\begin{theorem}
\label{thm:simplotope_everest}
$\setransform_{n,s}(\SoVertices(\simplotope_{n+1,s})\setminus\verticeszero{n+1}{s})=\SoVertices(\everestpolytopelong{n}{s})$, $\verticeszero{n+1}{s}\setminus\set{\vector{0}}$ is a basis of $\ker(\setransform_{n,s})$ (in particular, $\setransform_{n,s}\verticeszero{n+1}{s}=\set{\vector{0}}$), and $\vector{0}$ is contained in the interior of $\everestpolytopelong{n}{s}$. In particular, $\setransform_{n,s}\simplotope_{n+1,s}=\everestpolytopelong{n}{s}$.
\end{theorem}

\begin{proof}
Let
\begin{align*}
\vector{v}=\left(
\begin{matrix}
-\standard_s(j_1)\\
\vdots\\
-\standard_s(j_{n+1})
\end{matrix}
\right)\in\verticesminusone{n+1}{s}\setminus\verticeszero{n+1}{s},\quad
\vector{u}=\left(
\begin{matrix}
-\standard_s(j_0)\\
\vdots\\
-\standard_s(j_0)
\end{matrix}
\right)\in\verticeszero{n+1}{s},
\end{align*}
where $j_1,\ldots,j_{n+1}\in\set{0,\ldots,s}$ not all equal and $j_0\in\set{0,\ldots,s}$. First we note that any row of $\setransform_{n,s}$ is of the form
\begin{align*}
\vector{r}_{i,j}\ce(\underbrace{\vector{0},\ldots,\vector{0}}_{\overset{i-1}{\text{times}}},\standard_s(j),\underbrace{\vector{0},\ldots,\vector{0}}_{\overset{n-i}{\text{times}}},-\standard_s(j))\in\R^{(n+1)s}\simeq\R^{(n+1)\times s},
\end{align*}
where $i\in\set{1,\ldots,n}$, $j\in\set{1,\ldots,s}$, and $\vector{0}$ is the $s$-dimensional zero (row-)vector. For some $i$ and $j$, we can thus compute
\begin{align*}
\vector{r}_{i,j}\vector{v}&=\standard_s(j)(-\standard_s(j_i))-\standard_s(j)(-\standard_s(j_{n+1}))=\delta_{j_{n+1},j}-\delta_{j_i,j},
\end{align*}
where $\delta_{x,y}$ for $x,y\in\R$ is the Kronecker delta function, i.e. $\delta_{x,y}\in\set{0,1}$ and $\delta_{x,y}=1$ iff $x=y$. But then we get
\begin{align*}
\setransform_{n,s}\vector{v}&=
\left(
\begin{matrix}
\delta_{j_{n+1},1}-\delta_{j_1,1}&\ldots&\delta_{j_{n+1},s}-\delta_{j_1,s}\\
\vdots&&\vdots\\
\delta_{j_{n+1},1}-\delta_{j_n,1}&\ldots&\delta_{j_{n+1},s}-\delta_{j_n,s}
\end{matrix}
\right)
=
\left(
\begin{matrix}
\delta_{j_{n+1},1}&\ldots&\delta_{j_{n+1},s}\\
\vdots&&\vdots\\
\delta_{j_{n+1},1}&\ldots&\delta_{j_{n+1},s}
\end{matrix}
\right)-
\left(
\begin{matrix}
\delta_{j_1,1}&\ldots&\delta_{j_1,s}\\
\vdots&&\vdots\\
\delta_{j_n,1}&\ldots&\delta_{j_n,s}
\end{matrix}
\right)\\
&=
\left(
\begin{matrix}
\standard_s(j_{n+1})\\
\vdots\\
\standard_s(j_{n+1})
\end{matrix}
\right)
-
\left(
\begin{matrix}
\standard_s(j_1)\\
\vdots\\
\standard_s(j_n)
\end{matrix}
\right)
=
\left(
\begin{matrix}
-\standard_s(j_1)\\
\vdots\\
-\standard_s(j_n)
\end{matrix}
\right)
-
\left(
\begin{matrix}
-\standard_s(j_{n+1})\\
\vdots\\
-\standard_s(j_{n+1})
\end{matrix}
\right)
.
\end{align*}
If $j_{n+1}\neq0$ we thus get that $\setransform_{n,s}\vector{v}\in\verticesminusone{n}{s}-(\verticeszero{n}{s}\setminus\set{\vector{0}})$ and $\setransform_{n,s}\vector{v}\neq\vector{0}$ (since $j_1,\ldots,j_{n+1}$ are not all equal), hence $\setransform_{n,s}\vector{v}\in\verticesone{n}{s}\setminus\set{\vector{0}}$ by definition of $\verticesone{n}{s}$. If, on the other hand, $j_{n+1}=0$ then $\setransform_{n,s}\vector{v}\in\verticesminusone{n}{s}\setminus\set{\vector{0}}$. Altogether we can conclude
\begin{align*}
\setransform_{n,s}(\SoVertices(\simplotope_{n+1,s})\setminus\verticeszero{n+1}{s})
&=
\setransform_{n,s}(\verticesminusone{n+1}{s}\setminus\verticeszero{n+1}{s})
=
(\verticesone{n}{s}\cup\verticesminusone{n}{s})\setminus\set{\vector{0}}
=
\SoVertices(\everestpolytopelong{n}{s}).
\end{align*}
On the other hand we get
\begin{align*}
\vector{r}_{i,j}\vector{u}=\standard_s(j)(-\standard_s(j_0))-\standard_s(j)(-\standard_s(j_0))=0,
\end{align*}
hence $\setransform_{n,s}\verticeszero{n+1}{s}=\set{\vector{0}}$. Clearly $\verticeszero{n+1}{s}\setminus\set{\vector{0}}$ is linearly independent and $\vector{0}$ is an interior point of $\everestpolytopelong{n}{s}$.
\end{proof}

We are now ready to prove a formula for the Everest polytope.

\begin{maintheorem}
\label{maintheorem1}
The volume of the $(n,s)$-Everest polytope is given by
\[\vol(\everestpolytopelong{n}{s})= \frac{((n+1)s)!}{(ns)!(s!)^{n+1}}.\]
\end{maintheorem}
\begin{proof}
We want to apply the Lifting Theorem (Main Theorem~\ref{thm:lifting_theorem}) 
with $\polytope\gets\simplotope_{n+1,s}$, $\fold{\polytope}\gets\everestpolytopelong{n}{s}$,
$d\gets (n+1)s$ and $e\gets ns$.
However, a minor modification is needed, because the SE-transformation $\setransform_{n,s}$
is not a projection matrix. So, let $\tilde{\setransform}_{n,s}\ce
\left(\begin{matrix}
\setransform_{n,s}\\
0\;I_{s}
\end{matrix}\right)\in\R^{((n+1)s)\times((n+1)s)}$, $\Pi\ce\left(I_{ns}\;0\right)\in\R^{(ns)\times((n+1)s)}$, and let $\tilde{\simplotope}_{n+1,s}\ce\tilde{\setransform}_{n,s}\simplotope_{n+1,s}$ denote the transformed simplotope.
Clearly, $\vol(\tilde{\simplotope}_{n+1,s})=\vol(\simplotope_{n+1,s})$, $\Pi\tilde{\simplotope}_{n+1,s}=\everestpolytopelong{n}{s}$,
and the transformed spine points $\tilde{\spineset}_{n+1,s}\ce\tilde{\setransform}_{n,s}\spineset_{n+1,s}$ span the kernel of $\Pi$. Using Main Theorem~\ref{thm:lifting_theorem}
on $\tilde{\simplotope}_{n+1,s}$ and $\everestpolytopelong{n}{s}$, we obtain
\begin{eqnarray*}
\binom{(n+1)s}{ns} \vol(\tilde{\simplotope}_{n+1,s})=\vol(\tilde{\spineset}_{n+1,s}) \vol(\everestpolytopelong{n}{s}).
\end{eqnarray*}
Furthermore, $\vol(\tilde{\spineset}_{n+1,s})=\frac{1}{s!}$ since $\tilde{\spineset}_{n+1,s}=
\left\{
\left(\vector{0},\ldots,\vector{0},-\standard_s(j)
\right)
\mid
j\in\set{0,\ldots,s}
\right\}$. Moreover, $\simplotope_{n+1,s}$ is the $(n+1)$-fold product of simplices $\simplex_s$ spanned by $\vector{0}$ and $s$ unit vectors.
Hence, the volume of $\simplex_s$ is $1/s!$, and by Fubini's theorem,
\[\vol(\tilde{\simplotope}_{n+1,s})=\vol(\simplotope_{n+1,s})=\left(\vol(\simplex_s)\right)^{n+1} = \frac{1}{(s!)^{n+1}}.\]
Plugging in the formulas for $\vol(\tilde{\spineset}_{n+1,s})$ and $\vol(\tilde{\simplotope}_{n+1,s})$ into the formula given by the lifting theorem the claim follows by rearranging terms.
\end{proof}


\section{Projecting the Birkhoff polytope}
\label{sec:birkhoff}

We apply our findings to the family of \emph{Birkhoff polytopes} which are among the most studied
objects in polytope theory~\cite{Ziegler:2007,Pak:2000,LoeraLiuYoshida:2009}.
Because of the difficulty of determining their volumes, it is natural to ask whether
our Lifting Theorem yields new insights for Birkhoff polytopes as well.
In this section we use a (spinal) projection of the $n$-th Birkhoff polytope $\birkhoff_n$
to arrive at a projected polytope $\hat{\birkhoff}_n$ whose volume is linked to the
volume of $\birkhoff_n$ in a direct way. Our construction is not canonical, and other
variants of $\hat{\birkhoff}_n$ can be obtained, relating to the volume of the Birkhoff polytope
in a similar way.
The goal of this section is to exemplify one choice of $\hat{\birkhoff}_n$ for which
all boundary vertices can be effectively computed. 

Recall the definition of the $n$-th Birkhoff polytope $\birkhoff_n$ as the convex hull 
of all $(n\times n)$ permutation matrices, where we identify a permutation matrix
with a point in $\R^{n^2}$ by concatenating its rows. As noted by de~Wolff~\cite[Example 2.6]{Wolff:UP},
the Birkhoff polytope contains a spine: with
\begin{align*}
\vector{u}\ce
\left(
\begin{matrix}
0&1&0&\cdots&0\\
\vdots&\ddots&\ddots&\ddots&\vdots\\
0&\cdots&0&1&0\\
0&\cdots&\cdots&0&1\\
1&0&\cdots&\cdots&0
\end{matrix}
\right)
\in\R^{n\times n},
\end{align*}
the set
\begin{align*}
\birkhoffspine_n:=\set{\vector{u}^k\mid k\in\set{0,\ldots,n-1}}
\end{align*}
is a spine of $\birkhoff_n$ with $n$ points. 

While embedded in $\R^{n^2}$,
$\birkhoff_n$ is only $(n-1)^2$-dimensional; since our Lifting Theorem assumes the polytope
to be full-dimensional, our first step is to embed $\birkhoff_n$ in a subspace
where it is full-dimensional. A simple way to achieve this is to remove the last row and column of
each permutation matrix, thus identifying a vertex of $\birkhoff_n$ with a point
in $\R^{(n-1)^2}$. Note that a missing row and column of a permutation matrix can be recovered
by the remaining entries, so this vertex mapping is bijective.

Let $A_n$ denote the matrix that realizes the above linear map $\R^{n^2}\rightarrow\R^{(n-1)^2}$.
It can be shown that $A_n\birkhoff_n$ is indeed full-dimensional. In fact, we can specify
a matrix $B_n$ defining a linear map $\R^{(n-1)^2}\rightarrow\R^{n^2}$ in the opposite direction
such that the composition $B_nA_n$ maps $\birkhoff_n$ to a translation of $\birkhoff_n$.
Moreover, the square root of the determinant of $B_n^TB_n$ specifies the difference in volume of $\birkhoff_n$
and $A_n\birkhoff_n$. 

\begin{lemma}
\label{birkhoff_lemma_1}
$\vol(\birkhoff_n)=n^{n-1}\cdot\vol(A_n\birkhoff_n)$.
\end{lemma}

We refer to Appendix~\ref{app:details_on_birkhoff} for the (lengthy) definitions
of $A_n$ and $B_n$ and the proof of the statement.

\medskip

Our goal is to project the (stretched) Birkhoff polytope $A_n\birkhoff_n$ using its spine.
To simplify this process, we apply an affine transformation such that the spine contains the origin,
and the remaining $(n-1)$ spine points are the first $(n-1)$ coordinate vectors $\vector{e}_1,\ldots,\vector{e}_{n-1}$.
Clearly, there exists a matrix $C_n$ and a translation vector $\vector{b}_n$ that realize these properties.
Hence, we arrive at the polytope $C_nA_n\birkhoff_n+\vector{b}_n$ which has the spine $(\vector{0},\vector{e}_1,\ldots,\vector{e}_{n-1})$
and whose volume is related to $\vol(\birkhoff_n)$ by a factor of $n^{n-1}\abs{\det(C_n)}$ where

\begin{lemma}
\label{birkhoff_lemma_2}
$\abs{\det(C_n)}=1$.
\end{lemma}
Again, we refer to Appendix~\ref{app:details_on_birkhoff} for the definition of $C_n$ and $\vector{b}_n$ and the proof.

Because of the particular spine structure of the transformed polytope, the projection with respect to
the spine is simply described by removing the first $(n-1)$ coordinates of a point. Let $D_n$
denote the $(n-2)(n-1)\times (n-1)^2$-matrix realizing this projection and define
\[\fold{\birkhoff_n}:=D_n(C_nA_n\birkhoff_n+\vector{b}_n).\]
The volume of $\fold{\birkhoff_n}$ can now be expressed using Main Theorem~\ref{thm:lifting_theorem}.
Since the spine is spanned by $n-1$ coordinate vectors, its volume is $\frac{1}{(n-1)!}$. Therefore,
\begin{align*}
\binom{(n-1)^2}{n-1}\vol(C_nA_n\birkhoff_n+b_n)=\vol(\fold{\birkhoff_n})\frac{1}{(n-1)!}.
\end{align*}
Using Lemma~\ref{birkhoff_lemma_1} and  Lemma~\ref{birkhoff_lemma_2}, we arrive at the relation
\begin{align*}
\binom{(n-1)^2}{n-1}\vol(\birkhoff_n)=\vol(\fold{\birkhoff_n})\frac{1}{(n-1)!}n^{n-1}.
\end{align*}

Therefore, $\fold{\birkhoff_n}$
provides a new angle to study the volume of the Birkhoff polytope. 
We remark again that our construction has the advantage that the vertices
of $\fold{\birkhoff_n}$ can be easily generated since the matrices and vectors
of our construction are explicit (see Appendix~\ref{app:details_on_birkhoff}).
We list the $4!-4=20$ vertices of $\fold{\birkhoff}_4$ as an example (we set: $\overline{1}\ce-1$):

\begin{align*}
&\left(
\begin{matrix}
0&0&0\\
0&0&\overline{1}
\end{matrix}
\right),
\left(
\begin{matrix}
0&\overline{1}&1\\
0&1&0
\end{matrix}
\right),
\left(
\begin{matrix}
0&\overline{1}&1\\
0&0&0
\end{matrix}
\right),
\left(
\begin{matrix}
0&\overline{1}&0\\
0&1&0
\end{matrix}
\right),
\left(
\begin{matrix}
0&\overline{1}&0\\
0&0&1
\end{matrix}
\right),
\left(
\begin{matrix}
1&0&\overline{1}\\
0&\overline{1}&1
\end{matrix}
\right),
\left(
\begin{matrix}
1&0&\overline{1}\\
0&\overline{1}&0
\end{matrix}
\right),\\
&\left(
\begin{matrix}
0&0&0\\
1&0&0
\end{matrix}
\right),
\left(
\begin{matrix}
0&0&\overline{1}\\
1&0&0
\end{matrix}
\right),
\left(
\begin{matrix}
0&0&\overline{1}\\
0&0&1
\end{matrix}
\right),
\left(
\begin{matrix}
1&0&0\\
\overline{1}&0&0
\end{matrix}
\right),
\left(
\begin{matrix}
1&0&0\\
\overline{1}&\overline{1}&0
\end{matrix}
\right),
\left(
\begin{matrix}
0&1&0\\
0&0&\overline{1}
\end{matrix}
\right),
\left(
\begin{matrix}
0&1&0\\
\overline{1}&0&\overline{1}
\end{matrix}
\right),\\
&\left(
\begin{matrix}
0&0&0\\
\overline{1}&1&0
\end{matrix}
\right),
\left(
\begin{matrix}
0&0&0\\
0&\overline{1}&1
\end{matrix}
\right),
\left(
\begin{matrix}
\overline{1}&1&0\\
1&0&\overline{1}
\end{matrix}
\right),
\left(
\begin{matrix}
\overline{1}&1&0\\
0&0&0
\end{matrix}
\right),
\left(
\begin{matrix}
\overline{1}&0&1\\
1&0&0
\end{matrix}
\right),
\left(
\begin{matrix}
\overline{1}&0&1\\
0&1&0
\end{matrix}
\right).
\end{align*}


\section{Conclusions and further remarks}
\label{sec:conclusions_and_remarks}
Main Theorem~\ref{thm:lifting_theorem} combines several new results. 
It answers the question of the existence of a triangulation of a polytope 
under the constraint that a given subset of the vertices of the polytope 
must be contained in every maximal simplex of the triangulation. 
Furthermore, it characterizes all such triangulations 
and provides a method to compute one (or all) efficiently
from the lift of a star triangulation. 
Finally, it generalizes the well-known relation 
$\vol(AM)=\abs{\det(A)}\vol(M)$ where $M$ is a measurable subset of $\R^d$ 
and $A\in\R^{d\times d}$ to certain cases where $A$ is not a square matrix.
In particular, it allows us to express the volume of an object in $\R^d$ 
in terms of the volume its ``shadow'' in $\R^e$, and vice versa.

The shadow that a cube in $\R^3$ casts if the light shines parallel to any of its space diagonals is a regular hexagon. Assuming a cube with side length $\ell$, the theorem implies that the volume of the cube and the volume (area) of its shadow (the hexagon) differ by a factor of $\sqrt{3}/\ell$ which provides an alternative method to compute the volume of a hexagon from the volume of a cube. By lifting the ``complicated'' hexagon to a higher dimensional space it gains more symmetries and becomes the comparatively simple cube. In the same fashion, the complicated Everest polytope is the shadow of the simpler simplotope which allowed the computation of its volume in Main Theorem~\ref{maintheorem1}.

Starting with a polytope and a spine (with at least two points),
it is easy to determine the volume of the ``shadow'' 
with respect to the spine using our theorem.
On the other hand, there is no easy way to tell if a given shape 
is the shadow of some higher dimensional object 
and in the case of the Everest polytope, this is the interesting direction.
We pose the question of whether other polytopes (e.g., the Birkhoff polytope)
can be expressed as shadows of other polytopes.
For that purpose, it might be worthwhile to find general methods 
or at least good heuristics to determine 
if a complicated shape can be recognized as the shadow of some simpler object.

\paragraph*{Acknowledgements.} We thank Raman Sanyal and Volkmar Welker for
helpful discussions that have led to significant simplifications of
our exposition.

\bibliography{KerberTichyWeitzerConstrainedTriangulationBibliography}
\bibliographystyle{abbrv}

\begin{appendix}

\section{Details on the Birkhoff projection}
\label{app:details_on_birkhoff}
The vertices of the $n$-th Birkhoff polytope $\birkhoff_n\subseteq\R^{n^2}\simeq\R^{n\times n}$ are given by:
\begin{align*}
\set{\vector{v}\in\R^{n\times n}\mid \vector{v}\text{ is a permutation matrix}}.
\end{align*}
It is shown in \cite[Example 2.6]{Wolff:UP} that the subset
\begin{align*}
\birkhoffspine_n\ce\set{\vector{u}^k\mid k\in\set{0,\ldots,n-1}}
\end{align*}
of the vertices of $\birkhoff_n$, where
\begin{align*}
\vector{u}\ce
\left(
\begin{matrix}
0&1&0&\cdots&0\\
\vdots&\ddots&\ddots&\ddots&\vdots\\
0&\cdots&0&1&0\\
0&\cdots&\cdots&0&1\\
1&0&\cdots&\cdots&0
\end{matrix}
\right)
\in\R^{n\times n},
\end{align*}
forms a spine of $\birkhoff_n$. In the following we formally define the matrices $A_n$, $B_n$, $C_n$, and $D_n$ as well as the translation vector $\vector{b}_n$ that are used in the computations of Section~\ref{sec:birkhoff}. Let $m\ce n-1$ and
\begin{align*}
W_n&\ce
\left(
\begin{matrix}
1&\cdots&0&0\\
\vdots&\ddots&\vdots&\vdots\\
0&\cdots&1&0
\end{matrix}
\right)
\in\R^{m\times n},\;
X_n\ce
\left(
\begin{matrix}
0&\cdots&0\\
\vdots&&\vdots\\
0&\cdots&0
\end{matrix}
\right)
\in\R^{m\times n},\\
Y_n&\ce
\left(
\begin{matrix}
1&\cdots&0\\
\vdots&\ddots&\vdots\\
0&\cdots&1\\
-1&\cdots&-1
\end{matrix}
\right)
\in\R^{n\times m},\;
Z_n\ce
\left(
\begin{matrix}
0&\cdots&0\\
\vdots&&\vdots\\
0&\cdots&0
\end{matrix}
\right)
\in\R^{n\times m}.
\end{align*}
Furthermore, let
\bgroup
\allowdisplaybreaks
\begin{align*}
A_n&\ce
\left(
\begin{matrix}
W_n&X_n&\cdots&X_n&X_n\\
X_n&\ddots&\ddots&\vdots&\vdots\\
\vdots&\ddots&\ddots&X_n&\vdots\\
X_n&\cdots&X_n&W_n&X_n
\end{matrix}
\right)
=
\left(
\begin{matrix}
\begin{matrix}
1&\cdots&0&0\\
\vdots&\ddots&\vdots&\vdots\\
0&\cdots&1&0
\end{matrix}
&&\vector{0}&
\begin{matrix}
0&\cdots&0\\
\vdots&&\vdots\\
0&\cdots&0
\end{matrix}
\\
&\ddots&&\vdots\\
\vector{0}&&
\begin{matrix}
1&\cdots&0&0\\
\vdots&\ddots&\vdots&\vdots\\
0&\cdots&1&0
\end{matrix}&
\begin{matrix}
0&\cdots&0\\
\vdots&&\vdots\\
0&\cdots&0
\end{matrix}
\end{matrix}
\right)
\in\R^{m^2\times n^2},\\
B_n&\ce
\left(
\begin{matrix}
Y_n&Z_n&\cdots&Z_n\\
Z_n&\ddots&\ddots&\vdots\\
\vdots&\ddots&\ddots&Z_n\\
Z_n&\cdots&Z_n&Y_n\\
-Y_n&\cdots&\cdots&-Y_n
\end{matrix}
\right)
=
\left(
\begin{matrix}
\begin{matrix}
1&\cdots&0\\
\vdots&\ddots&\vdots\\
0&\cdots&1\\
-1&\cdots&-1
\end{matrix}
&&\vector{0}
\\
&\ddots&\\
\vector{0}&&
\begin{matrix}
1&\cdots&0\\
\vdots&\ddots&\vdots\\
0&\cdots&1\\
-1&\cdots&-1
\end{matrix}\\
\begin{matrix}
-1&\cdots&0\\
\vdots&\ddots&\vdots\\
0&\cdots&-1\\
1&\cdots&1
\end{matrix}
&\cdots&
\begin{matrix}
-1&\cdots&0\\
\vdots&\ddots&\vdots\\
0&\cdots&-1\\
1&\cdots&1
\end{matrix}
\end{matrix}
\right)
\in\R^{n^2\times m^2},\\
C_n&\ce
\left(
\begin{matrix}
\standard_{m^2}(2)\\
\vdots\\
\standard_{m^2}(n)\\
\standard_{m^2}(1)+\cdots+\standard_{m^2}(n)\\
-\standard_{m^2}(1)+\standard_{m^2}(n+1)\\
\vdots\\
-\standard_{m^2}(m^2-n)+\standard_{m^2}(m^2)
\end{matrix}
\right)
=
\left(
\begin{matrix}
0&1&0&\cdots&0&0&\cdots&\cdots&0\\
\vdots&\ddots&\ddots&\ddots&\vdots&\vdots&&&\vdots\\
\vdots&&\ddots&\ddots&0&\vdots&&&\vdots\\
0&\cdots&\cdots&0&1&0&\cdots&\cdots&0\\
1&\cdots&\cdots&\cdots&1&0&\cdots&\cdots&0\\
-1&0&\cdots&\cdots&0&1&0&\cdots&0\\
0&\ddots&\ddots&&&\ddots&\ddots&\ddots&\vdots\\
\vdots&\ddots&\ddots&\ddots&&&\ddots&\ddots&0\\
0&\cdots&0&-1&0&\cdots&\cdots&0&1
\end{matrix}
\right)
\in\R^{m^2\times m^2},\\
D_n&\ce
\left(
\begin{matrix}
\vector{0}&I_{m(m-1)}
\end{matrix}
\right)
\in\R^{m(m-1)\times m^2}\\
\vector{b}_n&\ce-\standard_{m^2}(n)\in\R^{m^2}.
\end{align*}
\egroup
Clearly, if $A_n$ is applied to a vertex $\vector{v}$ of $\birkhoff_n$ it simply drops the last row and column of $\vector{v}\in\R^{n\times n}\simeq\R^{n^2}$. The removed entries of $\vector{v}$ (which is a permutation matrix, i.e. contains exactly one '$1$' in each row and in each column with all other entries being equal to zero) can obviously be recovered from the remaining entries and an easy computation shows that this is exactly what $B_n$ does:
\begin{align*}
B_nA_n\vector{v}+\vector{a}_n=\vector{v}
\end{align*}
for all vertices $\vector{v}$ of $\birkhoff_n$, where
\begin{align*}
\vector{a}_n&\ce
\left(
\begin{matrix}
0&\cdots&0&1\\
\vdots&&\vdots&\vdots\\
0&\cdots&0&1\\
1&\cdots&1&-(n-2)
\end{matrix}
\right).
\end{align*}
We compute
\begin{align*}
{B_n}^TB_n&=
\left(
\begin{matrix}
{Y_n}^T&Z_n&\cdots&Z_n&-{Y_n}^T\\
Z_n&\ddots&\ddots&\vdots&-{Y_n}^T\\
\vdots&\ddots&\ddots&Z_n&-{Y_n}^T\\
Z_n&\cdots&Z_n&{Y_n}^T&-{Y_n}^T
\end{matrix}
\right)
\left(
\begin{matrix}
Y_n&Z_n&\cdots&Z_n\\
Z_n&\ddots&\ddots&\vdots\\
\vdots&\ddots&\ddots&Z_n\\
Z_n&\cdots&Z_n&Y_n\\
-Y_n&\cdots&\cdots&-Y_n
\end{matrix}
\right)\\
&=
\left(
\begin{matrix}
2{Y_n}^TY_n&{Y_n}^TY_n&\cdots&{Y_n}^TY_n\\
{Y_n}^TY_n&\ddots&\ddots&\vdots\\
\vdots&\ddots&\ddots&{Y_n}^TY_n\\
{Y_n}^TY_n&\cdots&{Y_n}^TY_n&2{Y_n}^TY_n
\end{matrix}
\right)=
\left(
\begin{matrix}
2J_n&J_n&\cdots&J_n\\
J_n&\ddots&\ddots&\vdots\\
\vdots&\ddots&\ddots&J_n\\
J_n&\cdots&J_n&2J_n
\end{matrix}
\right)
\in\R^{m^2\times m^2}
\intertext{with}
J_n&\ce{Y_n}^TY_n=I_m+1=
\left(
\begin{matrix}
2&1&\cdots&1\\
1&\ddots&\ddots&\vdots\\
\vdots&\ddots&\ddots&1\\
1&\cdots&1&2
\end{matrix}
\right)
\in\R^{m\times m}.
\end{align*}
Using the following lemma it is easy to show that $\det(J_n)=n$ and
\begin{align*}
\sqrt{\det({B_n}^TB_n)}=\sqrt{n^mn^m}=n^m.
\end{align*}
\begin{lemma}
Let $m,n\in\N$, $A\in\R^{m\times m}$, and
\begin{align*}
B&\ce
\left(
\begin{matrix}
2A&A&\cdots&A\\
A&\ddots&\ddots&\vdots\\
\vdots&\ddots&\ddots&A\\
A&\cdots&A&2A
\end{matrix}
\right)
\in\R^{(mn)\times(mn)}.
\end{align*}
Then $\det(B)=(n+1)^m\det(A)^n$.
\end{lemma}

\begin{proof}
Expansion of the first row (of blocks) of $B$ and induction on $n$.
\end{proof}

Furthermore, easy computations show that
\begin{align*}
\abs{\det(C_n)}=1
\end{align*}
and
\begin{align*}
C_nA_n\birkhoffspine_n+\vector{b}_n=\set{\standard_{m^2}(i)\mid i\in\set{0,\ldots,n-1}}.
\end{align*}
\end{appendix}
\end{document}